\documentclass[a4paper,11pt]{amsart}
\title{On classical upper bounds for slice genera}
\author{Peter Feller}
\address{Department of Mathematics, ETH Z\"urich, R\"amistrasse 101, 8092 Z\"urich, Switzerland}
\email{\myemail{peter.feller@math.ch}}
\author{Lukas Lewark}
\address{University of Bern, Mathematical Institute, Alpeneggstrasse 22, 3012 Bern, Switzerland}
\email{\myemail{lukas@lewark.de}}
\usepackage[latin1]{inputenc}
\usepackage{amssymb,amsmath,microtype,hyperref,caption,array,multirow,pifont,amsfonts,amsthm,bbm,graphicx,psfrag,verbatim,pstricks,mathrsfs,pifont,enumerate,enumitem,pdfpages,xcolor,thmtools,thm-restate,wrapfig,blindtext}
\usepackage[nameinlink,nosort]{cleveref}
\usepackage[pdftex,all,knot]{xy}

\Crefname{subsection}{Section}{Sections}
\Crefname{prop}{Proposition}{Propositions}
\Crefname{question}{Question}{Questions}
\crefname{equation}{}{}
\let\cref\Cref
\newcommand{\myemail}[1]{\href{mailto:#1}{#1}}

\hypersetup{colorlinks={true},linkcolor={black},citecolor={black},filecolor={black},urlcolor={black},pdfauthor={\authors},pdftitle={\shorttitle}}

\newcommand{\qua}{\hskip 0.4em \ignorespaces}
\def\arxiv#1{\relax\ifhmode\unskip\qua\fi
\href{http://arxiv.org/abs/#1}%
{\tt arXiv:\penalty -100\unskip#1}}

\def\MR#1{\relax\ifhmode\unskip\qua\fi
\href{http://www.ams.org/mathscinet-getitem?mr=#1}{\tt MR#1}}
\def\xox#1{\csname xx#1\endcsname}

\newcolumntype{x}[1]{>{\centering\arraybackslash\hspace{0pt}}p{#1}}
\makeatletter
\newcommand{\addresseshere}{%
  \enddoc@text\let\enddoc@text\relax
}
\makeatother

\newtheorem{thm}{Theorem}
\newtheorem{corollary}[thm]{Corollary}
\newtheorem{lemma}[thm]{Lemma}
\newtheorem{definition}[thm]{Definition}
\newtheorem{prop}[thm]{Proposition}
\newtheorem{claim}[thm]{Claim}
\newtheorem{question}
{Question}

\newtheorem{observation}[thm]{Observation}

\theoremstyle{remark}
\newtheorem{example}[thm]{Example}
\newtheorem{rmk}[thm]{Remark}

\newenvironment{Example}{\begin{example}\rm}{\end{example}}

\def\Q{\mathbb{Q}}

\def\Z{\mathbb{Z}}
\def\C{\mathbb{C}}

\def\ID{\mathbbm{1}}
\def\epsilon{\varepsilon}

\def\gtop{{g_{\rm{top}}}}
\def\gs{{g_{\rm{smooth}}}}
\def\ga{{g_{\rm{alg}}}}
\def\ua{{u_{\rm{alg}}}}
\def\gst{{g_{\Z}}}
\newcommand{\Alex}[1]{\ensuremath{\Delta_{#1}}}

\DeclareMathOperator{\rk}{rk}
\DeclareMathOperator{\rad}{rad}

\DeclareMathOperator{\coker}{coker}

\begin{document}
\maketitle
\begin{abstract}
We introduce a new link invariant called the algebraic genus, which gives an upper bound for the topological slice genus of links.
In fact, the algebraic genus is an upper bound for another version of the slice genus proposed here:
the minimal genus of a surface in the four-ball whose complement has infinite cyclic fundamental group.
We characterize the algebraic genus in terms of cobordisms in three-space, and explore the connections to other knot invariants
related to the Seifert form, the Blanchfield form, knot genera and unknotting.
Employing Casson-Gordon invariants, we discuss the algebraic genus as a candidate for the optimal upper bound for the topological slice genus
that is determined by the S-equivalence class of Seifert matrices. 

\end{abstract}
\section{Introduction}\label{sec:intro}
In this paper we introduce the notion of the
\emph{algebraic genus} of a link $L$ in $S^3$, denoted by $\ga(L)$. %
The main interest in $\ga$ is that it provides an upper bound for the \emph{$\Z$--slice genus} $\gst(L)$ of a link $L$---the smallest genus of an oriented connected properly embedded locally flat surface $F$ in the 4--ball $B^4$ with oriented boundary $L\subset \partial B^4 = S^3$ and $\pi_1(B^4\setminus F)\cong \Z$.
\begin{thm}\label{thm:gts<=ga}
For all links $L$, $\gst(L)\leq \ga(L)$.
\end{thm}

Postponing a more conceptual definition to \cref{sec:def}, we let the \emph{algebraic genus} $\ga(L)$ of a link $L$ with $r>0$ components be defined by
\[\min
\left\{\frac{m-r+1}{2}-n\ \middle| \begin{array}{l}\text{$L$ admits an $m\times m$ Seifert matrix of the form}\\
\begin{pmatrix}
A & * \\
* & *
\end{pmatrix},  \text{ where $A$ is a top-left } 2n\times 2n\\ \text{submatrix with } \det(tA-A^{\top})=t^n\end{array}\right\}.\]

We will establish (see \Cref{thm:character}) that the algebraic genus can be characterized as the \emph{3--dimensional cobordism distance} to the set of knots with Alexander polynomial~$1$,
and that it is related to a number of known knot invariants such as the algebraic unknotting number (see \Cref{thm:galequaleq2ga}).
However, in our opinion, what makes the algebraic genus most worth considering are the following two questions.
We conjecture that both of them have a positive answer.

\begin{question}\label{q:galg=gst}
Does $\ga(L)=\gst(L)$ hold for all links $L$, i.e.~is the inequality in \cref{thm:gts<=ga} an equality?
\end{question}
\begin{question}\label{q:galg=bestclassicalupperbound} Is the algebraic genus the best upper bound for the topological slice genus of a link $L$ determined by the S-equivalence class of the Seifert matrices of $L$? More precisely, is it true for all links $L$ that
$
\ga(L)=\max\{\gtop(L')\mid\text{The Seifert matrices of $L'$ are S-equivalent to those of $L$}\}
$?
\end{question}

\subsection{The disk embedding theorem and other context for the above questions}%

Freedman's celebrated disk embedding theorem~\cite{Freedman_82_TheTopOfFour-dimensionalManifolds,FreedmanQuinn_90_TopOf4Manifolds}
implies that a locally-flat $2$--sphere $S$ in $S^4$ is unknotted (i.e.~bounds an embedded locally flat 3--ball) if and only if its complement satisfies $\pi_1(S^4\setminus S)\cong\Z$~\cite[Theorem~11.7A]{FreedmanQuinn_90_TopOf4Manifolds}. This makes the study of surfaces with that fundamental group condition rather natural.

In the relative case of disks bounding knots, Freedman established the following~\cite{Freedman_82_TheTopOfFour-dimensionalManifolds}\cite[Theorem~11.7B]{FreedmanQuinn_90_TopOf4Manifolds},
which in fact is the only consequence of the disk embedding theorem that we will use in this text.
\begin{equation}\label{eq:Alex1<=>boundsstddisk}
\parbox{0.9\textwidth}{\centering A knot $K$ has Alexander polynomial 1 if and only if it bounds a properly embedded locally flat disk in the 4--ball whose complement has infinite cyclic fundamental group.}
\end{equation}
In terms of the invariants we introduce in this text, \Cref{eq:Alex1<=>boundsstddisk} may be written as
\[
\gst(K)=0 \Leftrightarrow \ga(K)=0.
\]
This gives a positive answer to the simplest case of \Cref{q:galg=gst}.

A positive answer to \Cref{q:galg=gst} in general would show that $\gst$ is a \emph{classical} link invariant in the sense of \cite{BorodzikFriedl_15_TheUnknottingnumberAndClassInv1}: a link invariant is classical if it only depends on the S-equivalence class of Seifert matrices of $L$. Such a simple---in particular 3--dimensional---characterization of $\gst$ would a priori be surprising. For example,
we note that such a characterization is impossible for the more extensively studied \emph{topological slice genus} $\gtop(L)$ of a link~$L$---the smallest genus of an oriented connected properly embedded locally flat surface $F$ in the 4--ball $B^4$ with oriented boundary $L\subset \partial B^4 = S^3$.
This is because there are pairs $K, K'$ of knots with the same Seifert form %
such that $K$ is topologically slice while $K'$ is not, i.e.~$\gtop(K) = 0$ and $\gtop(K') > 0$.
Such examples of knots $K$ and $K'$ were first found by Casson and Gordon using what are now known as \emph{Casson-Gordon invariants} \cite{cg,CassonGordon_86}.

In \cref{subsec:optimality}, we will see how
Gilmer's lower bounds derived from Casson-Gordon invariants~\cite{gilmer} can be used to obtain a partial answer to \Cref{q:galg=bestclassicalupperbound}.
Positive answers to both \Cref{q:galg=gst,q:galg=bestclassicalupperbound} would yield a rather satisfying understanding of the possible slice genera of links with a given S-equivalence class: the maximal upper bound in terms of Seifert forms is attained and it is equal to a version of the slice genus with a natural condition on $\pi_1$. This fits with the following important point about invariants that depend on more than just the S-equivalence class such as the Casson-Gordon invariants and $L^2$-signatures (as used by Cochran, Orr, and Teichner~\cite{COT}). %
Namely, these obstructions involve subtle questions concerning the extension of representations of $\pi_1$ of knot complements to $\pi_1$ of the complements of surfaces in $B^4$ bounding the knot; an issue that completely disappears when the latter complement has cyclic $\pi_1$.

\subsection{The algebraic genus via 3--dimensional cobordism distance}
Rather than in terms of Seifert matrices, $\ga(L)$ can also be characterized as the smallest genus of a cobordism in 3--space between $L$ and a knot with Alexander polynomial 1:
\begin{thm}\label{thm:character}
For all links $L$ with $r$ components, $\ga(L)$ equals the smallest genus among Seifert surfaces for links $L'$ with $r+1$ components such that the first $r$ components form $L$ and the last component forms a knot with Alexander polynomial 1.
\end{thm}
This characterization of $\ga(L)$ is the reason for naming the invariant ``algebraic genus'', in parallel to the algebraic unknotting number $\ua$ (see \cref{subsec:ua}): for a knot $K$, both $\ga(K)$ and $\ua(K)$ can be defined either purely in terms of the Seifert form, or as a $3$--dimensional distance (using the genus of Seifert surfaces and unknotting, respectively) to knots that have Alexander polynomial 1.
The name ``algebraic slice genus'', on the other hand, would be more fitting for Taylor's invariant (see \cref{subsec:taylor}). %

\subsection{The algebraic genus and other knot invariants}
We summarize the relation between $\ga$ and other knot invariants in \cref{fig:invdiag}.
By $g$ and $\gs$ we respectively denote the three-dimensional genus and the smooth slice genus, neither of which is classical (i.e.~determined by the S-equivalence class).
\begin{figure}[h]
\centering
\includegraphics{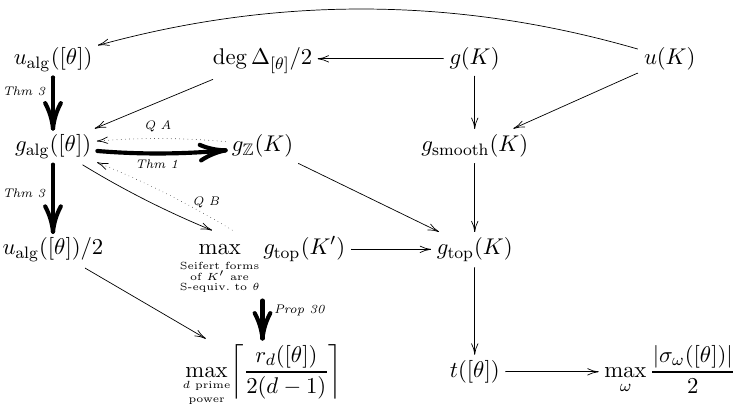}
\caption{Diagrammatic summary of relations of $\ga$ to other knot invariants, in homage to \cite[Figure 2]{BorodzikFriedl_15_TheUnknottingnumberAndClassInv1}.
Arrows $a \rightarrow b$ indicate that the inequalities $a \geq b$ hold for all knots.
A dotted arrow means that the status of the respective inequality is open.
If there is no directed path from $a$ to $b$, it means that $a \geq b$ is known to be false for some knots.
Thick arrows indicate original results of this article.
We write $[\theta]$ (the S-equivalence class of a Seifert form $\theta$ of $K$) rather than $K$ as an argument %
to indicate that an invariant is classical, i.e.~depends only on $[\theta]$.}\label{fig:invdiag}
\end{figure}
Some of the considered invariants, notably the algebraic unknotting number and Taylor's invariant,
are currently only defined for knots. One might expect those invariants and their relations to $\ga$
to generalize to multi-component links, for which $\ga$ is naturally defined.
However, such generalizations need not be straight-forward.
This is the reason we consider knots rather than multi-component links in \Cref{subsec:ua,subsec:taylor,subsec:optimality}.

\subsection{The algebraic genus and the algebraic unknotting number}\label{subsec:ua}
The algebraic unknotting number $\ua$ of a knot, introduced by Murakami~\cite{Murakami}, is the optimal classical lower bound
for the unknotting number $u$. We prove the following inequality between $\ua$ and $\ga$:
\begin{restatable}{thm}{uathm}\label{thm:galequaleq2ga}
For all knots $K$,
$
\ga(K)\leq\ua(K)\leq 2\ga(K).
$
\end{restatable}
Let $\Alex{K}$ denote the Alexander polynomial of $K$. We understand its \emph{degree} $\deg(\Alex{K}(t))$ to be the breadth of $\Alex{K}$; e.g.~$2$ for the trefoil.
Then, using that $2\ga(K)\leq \deg(\Alex{K}(t))$, the above theorem yields
\begin{corollary}\label{cor:ualeqdegAlex}
For all knots,  $\ua(K)\leq \deg(\Alex{K}(t))$.
\end{corollary}
This answers a question of Borodzik and Friedl, who had previously shown $$\ua(K)\leq \deg(\Alex{K}(t))+1$$
and asked whether that bound could be sharpened.
We use their characterization of $\ua$ in terms of the Blanchfield pairing (cf. \cite[Theorem~2]{BorodzikFriedl_14_OnTheAlgUnknottingNr}, \cite[Lemma~2.3]{BorodzikFriedl_15_TheUnknottingnumberAndClassInv1})
for the proof of the second inequality of \cref{thm:galequaleq2ga}.

\subsection{The algebraic genus and Taylor's invariant}\label{subsec:taylor}
The algebraic genus can be understood as a measure of how much a knot fails to have Alexander polynomial $1$. Indeed, $\ga(K)=0$ if and only if $\Alex{K} = 1$.
Taylor introduced a knot invariant $t(K)$ that generalizes algebraic sliceness~\cite{taylor}: a knot $K$ is algebraically slice, i.e.~has a metabolic Seifert form, if and only if $t(K)=0$.
Explicitly, if $\theta: \Z^{2n}\times\Z^{2n}\to\Z$ is a Seifert form of a knot $K$, then $t(K)$ is defined as $n$ minus the maximal rank of a totally isotropic subgroup of $\Z^{2n}$.
Taylor's invariant provides a lower bound for the topological slice genus\footnote{Taylor's original proof~\cite{taylor} is in the PL category, but the statement is known by experts to hold in the locally-flat category. Not being aware of a detailed reference, let us briefly outline a proof for $t(K)\leq\gtop(K)$.

The crucial ingredient is that every locally-flat closed oriented surface in $S^4$ is the boundary of a compact oriented locally-flat 3-dimensional submanifold of $S^4$. This can be deduced from topological transversality as established in~\cite[Chapter 9]{FreedmanQuinn_90_TopOf4Manifolds} (see also~\cite[page~XXI]{Ranicki}).

Now, for a fixed knot $K$, let $S$ be a locally-flat slice surface in $B^4$ realizing $\gtop(K)$
and pick a Seifert surface $F$ for $K$ in $S^3=\partial B^4$. Viewing $B^4$ as a subset of $S^4$, we may push the interior of $F$ away from $B^4$ such that $F\cup S$ forms a closed locally flat surface in $S^4$ with $K=F\cap S=S^3\pitchfork (F\cup S)$.
By the crucial fact, we find a locally flat 3-manifold $M\subset S^4$ with boundary $F\cup S$.
Making $M$ transverse to $S^3$ (while keeping $\partial M$ transverse to $S^3$)
and intersecting with $B^4$ yields a locally-flat 3-manifold in $B^4$ whose boundary is the union of a (possibly non-connected) Seifert surface of $K$ in $S^3$ and an isotopic copy of~$S$.
From there, one may follow a standard proof, as e.g.\ given in \cite[Chapter~8]{Lickorish_97} for the case $\gtop(K)=0$.}, which is indeed the optimal classical bound.
In particular, Taylor's bound subsumes the bounds given by the Levine-Tristram signatures~$\sigma_{\omega}$.

Since $t(K)\leq \gtop(K)\leq \ga(K)$, it would be of interest to relate $t(K)$ and $\ga(K)$. %
It appears that aside from $t(K) \leq \ga(K)$
the two invariants are rather independent. However, we can prove the following:
if a genus 2 fibered knot $K$ is algebraically slice (i.e.~$t(K)=0$), then $\ga(K)$ (and thus also the topological slice genus of $K$) is at most 1; compare \Cref{prop:fibred}.
In contrast, such results are not available for knots with Alexander polynomials of higher degree.
For example, there exist algebraically slice knots $K$ with monic Alexander polynomial of degree $6$ and $\ga(K) = 3$; compare \cref{ex:r2}.

\subsection{Towards optimality of the algebraic genus as slice genus bound}\label{subsec:optimality}
Taylor's lower bound to the slice genus is known to be the best classical lower bound for knots (see \cref{subsec:taylor});
that is to say, every Seifert form $\theta$ is realized by a knot whose slice genus equals $t(\theta)$.
\Cref{q:galg=bestclassicalupperbound} is the analogous question about classical \emph{upper} bound for the topological slice genus.
As a first step towards determining the best classical upper bound for the topological slice genus of
knots---for which the algebraic genus is a candidate---we prove in \Cref{thm:cg}
that every Seifert form $\theta$ of a knot is realized by a knot $K$ with
\[
\phantom{\max_{K\in\mathcal{S}_{\theta}}}
\gtop(K) \geq \max_{\substack{d\text{ {prime}}\\\text{{power}}}} \biggl\lceil \frac{r_d(\theta)}{2(d - 1)} \biggr\rceil.
\]
Here, $r_d(\theta)$ denotes the minimum number of generators of the first integral homology group of the $d$--fold branched cover
of a knot $K'$ realizing $\theta$ (note that $r_d$ only depends on $\theta$).
The relevant knots in the proof are constructed via infection, following Livingston~\cite{livingstonseifert}.

\subsection{Calculations of the algebraic genus and its role as upper bound for the slice genus}
It is a virtue of $\ga$ that upper bounds for it can be explicitly calculated using Seifert matrix manipulation.
Previous work by Baader, Liechti, McCoy and the
authors~\cite{Feller_15_DegAlexUpperBoundTopSliceGenus,BaaderLewark_15_Stab4GenusOFAltKnots,FellerMcCoy_15,BaaderFellerLewarkLiechti_15,mccoylewark,LiechtiBraids}
used this method (without any focus on the algebraic genus itself) to determine upper bounds for the topological slice genus of various classes of links.
Due to \Cref{thm:gts<=ga}, all of those results in fact give upper bounds for the $\mathbb{Z}$--slice genus.

Although no general algorithm is known, in many cases a combination of calculable upper and lower bounds for the algebraic genus determines it completely.
For example, the algebraic genus has been calculated for all prime knots with 11 crossings or less \cite{mccoylewark}.

\subsection{Structure of the paper}
In \Cref{sec:def}, the algebraic genus is defined, first examples are given and basic results are proven, as well as a result about the stable algebraic genus.
\Cref{sec:3d} contains the proof for the alternative three-dimensional characterization of $\ga$ given in \cref{thm:character}.
\Cref{thm:gts<=ga} and \Cref{thm:galequaleq2ga} on the $\Z$--slice genus and the algebraic unknotting number are proven in \Cref{sec:st,sec:ua}, respectively.
\Cref{sec:fibred} is concerned with the algebraic genus of knots with monic Alexander polynomial, and contains the proof of \Cref{prop:fibred}.
In \Cref{sec:optimality}, optimality of slice genus bounds is discussed and \Cref{thm:cg} is proven.
The paper concludes with the short \Cref{sec:reformulate}, in which previously known results are reformulated in terms of the algebraic genus.
\subsection{Acknowledgments}
We thank Danny Ruberman for pointing us to \cite{taylor}.
We thank Sebastian Baader and Livio Liechti for valuable inputs; in particular, concerning \cref{prop:fibred}.
We thank Mark Powell for comments on a first version of this paper, and the referee for
helpful suggestions.
Both authors gratefully acknowledge support by the SNSF and thank the MPIM Bonn for its hospitality.

\section{The algebraic genus---basic definitions and properties}\label{sec:def}

\subsection{Definitions}
We consider links, by which we mean ~smooth oriented non-empty closed 1-dimensional submanifolds of $S^3$.
We define the algebraic genus of a %
link $L$ using the Seifert form defined on $H_1(F;\Z)\cong \Z^{2g+r-1}$, where $F$ is a genus $g\geq0$ Seifert surface with boundary the $r>0$ component link $L$. By a \emph{Seifert surface} for a link $L$, we mean an oriented connected embedded surface in $S^3$ with boundary $L$.
The \emph{genus} $g(L)$ of $L$ is the minimum genus of a Seifert surface of $L$.

Let us start with some notations on bilinear forms (which we will readily use for Seifert forms).
For integers $g\geq 0$ and $r\geq 1$, let $\theta$ be a bilinear form on an abelian group $H\cong\Z^{2g+r-1}$ such that its
antisymmetrization, denoted by $\theta-\theta^\top$, satisfies the following:
the radical $\rad_{\theta-\theta^\top}$ of $\theta-\theta^\top$---the subgroup of elements that pair to 0 with all other elements---is isomorphic to $\Z^{r-1}$ as a group and the form induced by $\theta-\theta^\top$ on $H/\rad_{\theta-\theta^\top}$ has determinant 1.
These are precisely the bilinear forms that arise as Seifert forms of genus $g$ Seifert surfaces of links with $r$ components.
If $M$ is a matrix representing such a form $\theta$, we call
\[
t^{-g}\cdot \det (t\cdot M - M^{\top}) \in \Z[t^{\pm 1}]
\]
the \emph{Alexander polynomial} of $\theta$, and denote it by $\Alex{\theta}$.
This is independent of the choice of basis and hence indeed a well-defined; in fact, it is invariant under S-equivalence.
We call a subgroup $U\subseteq H\cong\Z^{2g+r-1}$ \emph{Alexander-trivial} if $\det(t\cdot M - M^{\top})$ is a unit in $\Z[t^{\pm 1}]$ for a matrix $M$ representing $\theta|_U$.
One obtains
\[
\det(M-M^{\top}) = 1
\]
by substituting $t = 1$.
It follows that $U$ is a summand of $H\cong \mathbb{Z}^{2g + r - 1}$, and the rank of $U$ is even and at most $2g$.
Furthermore, $U$ is Alexander-trivial if and only if $\theta|_U$ is the Seifert form of a knot $K$ with $\Alex{K}=\Alex{\theta|_U} = 1$.

Suppose $2d$ is the maximal rank of an Alexander-trivial subgroup for a bilinear form $\theta$.
We define $\widetilde{\ga}(\theta)$
to be $\widetilde{\ga}(\theta) = g - d$ and
we define $\ga(\theta)$ to be the minimum $\widetilde{\ga}(\eta)$, where $\eta$ ranges over all
forms that are S-equivalent to $\theta$.
\begin{definition}
\label{def:galg}For all %
links $L$, we define the \emph{algebraic genus $\ga(L)$ of} $L$ to be
\[
\ga(L)=\min\left\{\ga(\theta)\;\middle|\begin{array}{c}\;\text{$\theta$ is the Seifert form of}\\\text{a Seifert surface for $L$}\end{array}\right\}.
\]
\end{definition}
Clearly, if $L$ is a link and $\theta$ some fixed Seifert form of $L$, then
\begin{equation}\label{eq:gagaga}
\ga(\theta) \leq \ga(L)\leq \widetilde{\ga}(\theta).
\end{equation}
We will prove in \Cref{prop:gaS-equivalenceclass} that $\ga(\theta) = \ga(L)$. But whether the second inequality of \Cref{eq:gagaga} is an equality remains an open question.

Note that reversing the orientation of all components of $L$ or taking the mirror image of $L$ does not change $\ga(L)$ since the Alexander-trivial subgroups with respect to $\theta$, $\theta^{\top}$, and $-\theta$ are the same.
\subsection{More on Alexander-trivial subgroups}

In practice, establishing that a subgroup $U\subseteq H$ is Alexander-trivial may be done by finding a basis of $U$ with respect to which $\theta|_U$ is given by a matrix $M$ of the form
\begin{equation}\label{eq:MforAlexTriv}\left(\begin{matrix}
0 & \ID + P\\
L & Q
\end{matrix}\right),\end{equation}
where $0$, $\ID$, $P$, $L$ and $Q$ denote square matrices of half the dimension of $M$ that are zero, the identity, lower triangular with zeros on the diagonal, upper  triangular with zeros on the diagonal, and arbitrary, respectively. For this we note that, if a $2n\times 2n$ matrix is of the form \cref{eq:MforAlexTriv}, then $\det(t\cdot M - M^{\top})=t^{n}$. The following lemma implies that Alexander-triviality of a subgroup can always be established by finding such a basis.
\begin{lemma}\label{lem:stdformofalextriv}
If a Seifert form $\theta$ on $H\cong\Z^{2g}$ has Alexander polynomial 1, then there exists a basis of $\Z^{2g}$ with respect to which $\theta$ is given by a matrix of the form
$\left(\begin{matrix}
0 & \ID + P\\
P^{\top} & 0
\end{matrix}\right),$
where $0$, $\ID$, and $P$ denote $g\times g$ matrices that are zero, identity, and upper triangular with zeros on the diagonal, respectively.
\end{lemma}
For general $\theta$, there is no basis such that $P$ is the zero matrix, since the rank of a matrix of $\theta$ is an invariant of the form.
Note that a significantly stronger statement holds:
there are knots which have Alexander-polynomial~$1$, yet do not admit a Seifert matrix as above with $P$ the zero matrix \cite{GaroufalidisTeichner_04_OnKnotswithtrivialAlex}.
\begin{proof}[Proof of \Cref{lem:stdformofalextriv}]
By a calculation provided in~\cite[Lemma~6 and Remark~7]{Feller_15_DegAlexUpperBoundTopSliceGenus}
there is a basis such that the corresponding $(2g\times 2g)$ matrix is of the form $M=\left(\begin{matrix}
0 & \ID + P\\
P^{\top} & Q
\end{matrix}\right),$
where $Q$ is some $g\times g$ matrix. The statement follows by applying the following base change
\[\left(\begin{matrix}
\ID & 0\\
-N & \ID
\end{matrix}\right)M\left(\begin{matrix}
\ID & -N^{\top}\\
0 & \ID \end{matrix}\right)=\left(\begin{matrix}0 & \ID + P\\
P^{\top} & Q-N(\ID+P)-P^{\top}N^{\top}\end{matrix}\right),\]
where $N$ is the unique $g\times g$ matrix that satisfies the equation $Q=N+NP+(NP)^{\top}$.
To be explicit, $N$ is inductively given as follows. Set $N_{11}=Q_{11}$. For the induction step, we fix $\ell\in\{2,3,\ldots,2g-1\}$ and assume $N_{ij}$ is defined whenever $i+j\leq \ell$, and thus we can set
\[
\pushQED{\qed}
N_{ij}=Q_{ij}-\sum_{k=1}^{j-1}(N_{ik}P_{kj})-\sum_{k=1}^{i-1}(N_{jk}P_{ki})\quad\text{whenever }i+j=\ell+1.
\qedhere
\popQED
\]
\renewcommand{\qedsymbol}{}
\end{proof}
\vspace{-2mm}
\subsection{Examples}%
\begin{Example}\label{Ex:12a908}
We consider the  knot $K_{12a908}$ and one of its Seifert matrices
\[
\footnotesize
M=
\begin{pmatrix}
-1&0&-1&0&0&-1\\
-1& 1&-1&1&1&0\\
0&0&-2&0&0&-2\\
1&0&1&-2&0&1\\
0&0&0&1&1&0\\
0& 0& -1& 0& 0& -2
\end{pmatrix}
\]
as provided by KnotInfo~\cite{knotinfo}. There exists an Alexander-trivial subgroup of rank $4$ in $\Z^6$ with respect to the bilinear form represented by $M$. Indeed,
\[
B =
\begin{pmatrix}
3  & 0  & -2&  1   \\
-5 & 1  & 2 &  0   \\
2  & -1 & -1&  0   \\
-2 & 0  & -1&  2   \\
6  & 0  & 0 &  -2  \\
1  & 0  & -1&  1
\end{pmatrix}
\quad\Rightarrow\quad
B^{\top}\cdot M \cdot B =
\begin{pmatrix}
0 & 0 & 1 & -3  \\
0 & 0 & 0 & 1  \\
0 & 0 & -5 & 8  \\
2 & 0 & 5 & -8  \\
\end{pmatrix},
\]
which shows that the columns of $B$ are the basis of an Alexander-trivial subgroup of rank $4$ in $\Z^6$ (note that $B^{\top}MB$ is of the form \cref{eq:MforAlexTriv}).
Furthermore, no Seifert form of $K_{12a908}$ has an Alexander-trivial subgroup of full rank, since the Alexander polynomial of $K_{12a908}$ is different from 1 (it is $4t^3 - 22t^2 + 55t - 73 + 55t^{-1} - 22t^{-2} + 4t^{-3}$). Thus, $\ga(K_{12a908})=1$ by \cref{def:galg}.
In fact, $|\sigma(K_{12a908}) / 2|=1$, where $\sigma(K)$ denotes the signature of the knot $K$, and so
\[\left|\frac{\sigma(K_{12a908})}{2}\right|=\gtop(K_{12a908})=\gst(K_{12a908})=\ga(K_{12a908})=1.\]

The genus of $K_{12a908}$ is 3 (since the degree of the Alexander polynomial of $K_{12a908}$ is 6), the smooth slice genus is 2 (by an argument based on Donaldson's diagonalization theorem;
compare~\cite{mccoylewark}), and the algebraic unknotting number is $2$ \cite{BorodzikFriedl_15_TheUnknottingnumberAndClassInv1,knotorious}.
Therefore, there is no immediate way via the smooth slice genus or the algebraic unknotting number to find that $\gtop(K_{12a908})=1$; while the above calculation of $\ga$ is quite explicit.
\end{Example}
{\makeatletter
\let\par\@@par
\par\parshape0%
\everypar{}
\begin{wrapfigure}{r}{51mm}%
\centering
\vspace{-2.8cm}%
\includegraphics[width=3.5cm]{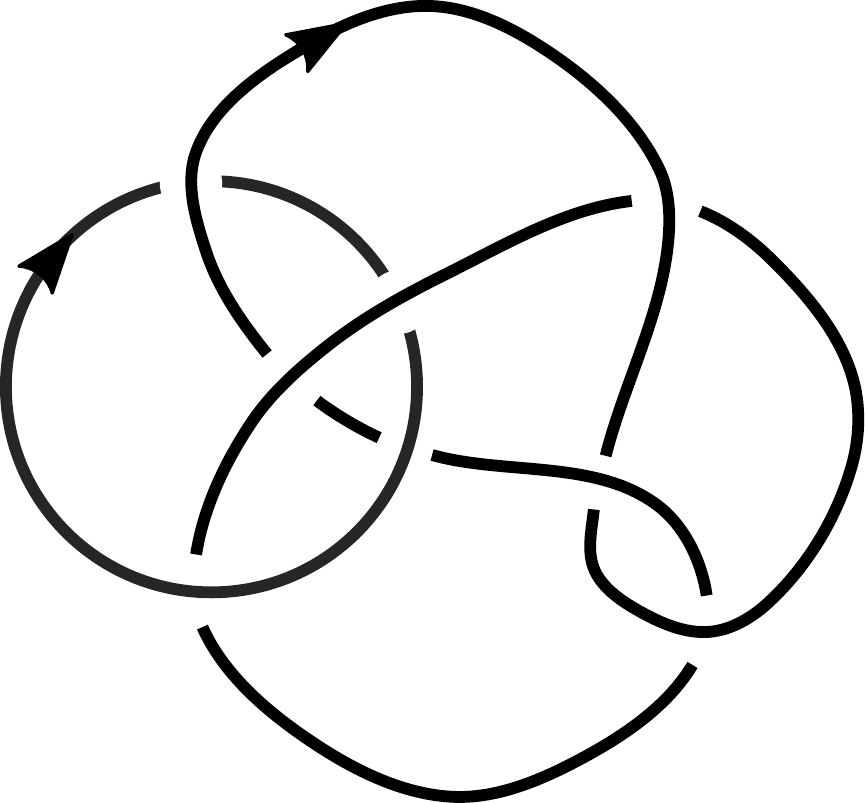}%
\caption[{The link L8n2 (diagram from KnotInfo \cite{knotinfo}).}]{The link L8n2 \linebreak (diagram from KnotInfo \cite{knotinfo}).}%
\label{fig:linkex}%
\end{wrapfigure}%
\noindent\begin{minipage}{7cm}%
\begin{Example}
Let us calculate the algebraic genus of the two-component link $L$ shown below in \Cref{fig:linkex}.
Seifert's algorithm gives a Seifert surface of genus 2 with a Seifert matrix $M$:\\[-1ex]%
\[
M =
\begin{pmatrix}
-1& 1 & 0 & 0 & 0 \\
0 & 0 & 1 & 0 & 0 \\
0 & 1 & 0 & 0 & 1 \\
0 & 0 & 0 & 0 & 0 \\
0 & 0 & 0 & -1& 0
\end{pmatrix}.
\medskip
\]
\end{Example}
\end{minipage}
\par}%
The standard basis vectors $e_1, e_2, e_4, e_5$ form an Alexander-trivial subgroup for $M$, and so $\ga(M) = 0$.
Consequently, all links with the Seifert matrix $M$ are topologically (weakly) slice.
The particular link $L$ turns out to be actually smoothly slice, and has three-genus $1$; note $M$ can be destabilized
by removing $e_4$ and $e_5$.
\subsection{Basic properties of the algebraic genus}
\cref{def:galg} is made such that for all links $L$ the inequality
\begin{equation}\label{eq:gt<=ga}
\gtop(L) \leq \ga(L)
\end{equation}
follows immediately from Freedman's \cref{eq:Alex1<=>boundsstddisk} and the following proposition,
which is proven in detail in~\cite[Proof of Prop.~3]{BaaderFellerLewarkLiechti_15} (compare also~\cite[Proposition~2]{Feller_15_DegAlexUpperBoundTopSliceGenus}).
For the sake of completeness, we nevertheless include a concise version of the proof below.
\Cref{prop:sepcurve} will be used in the proof of \cref{thm:gts<=ga}, which subsumes \cref{eq:gt<=ga}.
\begin{prop}\label{prop:sepcurve}
Let $L$ be a link, $F$ a genus $g$ Seifert surface for $L$, and $\theta$ the Seifert form on $H_1(F;\Z)$.
If $U$ is an Alexander-trivial subgroup of $H_1(F;\Z)$ of rank $2d$, then there exists a separating simple closed curve $K$ on $F$ with the following properties:
\begin{itemize}
  \item the curve $K$ (viewed as a knot in $S^3$) has Alexander polynomial 1,
  \item $F \setminus K = F_1 \sqcup F_2$ such that $\overline{F_1}$ is of genus $d$ with boundary $K$.
\end{itemize}
\end{prop}
\begin{proof}[Sketch of the proof]
As discussed at the beginning of the section, $U$ is a summand, i.e.\ $H_1(F;\Z) = U \oplus V$.
One may choose a separating simple closed curve $K'$ on $F$ such that $F \setminus K' = F_1' \sqcup F_2'$,
$\partial F_1' = K'$, and $\rk H_1(F_1'; \Z) = 2d$.
It turns out that there is a group automorphism $\varphi': H_1(F;\Z) \to H_1(F;\Z)$ with $\varphi'(U) = H_1(F_1';\Z)$ that preserves
the intersection form $\theta - \theta^{\top}$ and that maps homology classes given by components of $\partial F$ to homology classes given by components of $\partial F$.
The group automorphism $\varphi'$ is realized as the action of a diffeomorphism $\varphi$ of $F$ since the mapping class group of $F$ surjects onto the symplectic group. Take $K = \varphi(K')$ and $F_i = \varphi(F_i')$.
Then $\overline{F_1}$ is a genus $d$ Seifert surface of $K$ and
the corresponding Seifert form is given by $\theta|_{U}$, hence $\Alex{K}=1$.
\end{proof}

The next proposition shows that $\ga(L)$ only depends on the S-equivalence class of  Seifert forms of $L$.
\begin{prop}\label{prop:gaS-equivalenceclass}
For all links $L$,
\[\ga(L)=\min\left\{\widetilde{\ga}(\theta)\;\middle|\begin{array}{c}\;\text{$\theta$ is a bilinear form that is S-equivalent}\\\text{ to a Seifert form (and thus all) for $L$}\end{array}\right\}.\]
\end{prop}
\begin{proof}
Clearly $\geq$ holds, since the minimum ranges over a larger class of bilinear forms.
For the other direction, we recall that any bilinear form S-equivalent to a Seifert form, can be stabilized to become a Seifert form. Indeed, let $\theta_1$ be a Seifert form for $L$ and $\theta_2$ any bilinear form S-equivalent to $\theta$. There exists a bilinear form $\theta$ that arises as the stabilization of both $\theta_1$ and $\theta_2$. Since stabilizations of a Seifert form can be realized geometrically by a stabilization of the corresponding Seifert surface, the bilinear form $\theta$ arises as a Seifert form.
Now, the statement follows from the following lemma about stabilizations.
\end{proof}
\begin{lemma}\label{lemma:stabdecreasesga}
Let $\theta$ be a bilinear form arising as Seifert form of a link, and let $\eta$ be obtained from $\theta$ by a stabilization. Then
$\widetilde{\ga}(\eta)\leq\widetilde{\ga}(\theta)$.
\end{lemma}
\begin{proof}
For the bilinear form $\theta$ on $\Z^{2g+r-1}$, let $U\subseteq \Z^{2g+r-1}$ be an Alexander-trivial subgroup of maximal rank, and denote this rank by $2d$. We view the stabilization $\eta$ of $\theta$ as a bilinear form on
 $\Z^{2g+r-1}\oplus\Z^2$ such that with respect to the standard basis $\eta$ is given by
\begin{equation*}\left(
{\begin{tabular}{ccc|cc}&&&&0\\
&$M$&&&0\\
&&&$v$&$\vdots$\\
&&&&$0$\\
\hline
\rule{0pt}{2.2ex}&$v^{\top}$&&0&1%
\\
0&$\cdots$&0&0%
&0\\
\end{tabular}}
\right),\end{equation*}
where $M$ represents $\theta$ and $v$ is some element of $\Z^{2g+r-1}$. The statement follows from the fact that $U\oplus\Z^2$ is an Alexander-trivial subgroup of $\Z^{2g+r-1}\oplus\Z^2$ with respect to $\eta$.
\end{proof}
\subsection{Lower bounds for the algebraic genus of links}
For knots, prominent lower bounds for the algebraic genus come from the ranks of branched covers (see \Cref{fig:invdiag}),
and from Taylor's invariant; the latter statement is evident from the definitions, because an Alexander-trivial subgroup of rank $2d$ contains a totally isotropic subgroup of rank $d$. In this subsection we investigate what can be proved for multi-component links.

For this purpose, let
$\eta(L)$ denote the nullity of a link $L$, defined as $\dim (\rad_{(\theta + {\theta}^{\top})\otimes\Q})$ for any Seifert form $\theta$ of $L$ (i.e.~interpret $\theta + {\theta}^{\top}$ as a bilinear form over the rationals and take the dimension of its radical).
Let $r_2(L)$ be the minimum number of generators of $H_1(M_2(L);\Z)$,
the first integral homology group of the double branched covering $M_2(L)$.

\begin{prop}\label{prop:bounds}
If $L$ is an $r$--component link, then:
\begin{align}
|\sigma(L)| + \eta(L) -r + 1 & \leq 2\ga(L), \tag{i}\\
r_2(L) & \leq 2\ga(L).\tag{ii}
\end{align}
\end{prop}

Note that both lower bounds for $2\ga(L)$ are classical, and additive with respect to the connected sum along arbitrary components.
\begin{proof}
\textbf{(i):}
Of course, (i) follows because $2\gtop(L)$ is greater than or equal to the left hand side, and less than or equal to the right hand side,
but there is also a more direct reason:
pick a Seifert form $\theta: \Z^{2g + r - 1}\times\Z^{2g + r - 1}\to\Z$
of $L$ with $\widetilde{\ga}(\theta) = \ga(L)$. Denote by $n_{\pm}$ the indices of inertia of the symmetrization of $\theta$,
so that $n_+ - n_- = \sigma(L)$ and $n_+ + n_- + \eta(L) = 2g + r - 1$.
There is an Alexander-trivial subgroup of rank $2(g - \ga(L))$, which implies that both $n_+$ and $n_-$
are greater than or equal to $g - \ga(L)$. Hence $|\sigma(L)| \leq 2\ga + r - 1 - \eta(L)$.

\textbf{(ii):}
Note that if $A$ is a matrix for $\theta$, then $H_1(M_2(L);\Z)$ is isomorphic to the cokernel of $A + A^{\top}$.
By the classification of finite abelian groups, there is a prime $p$ such that $\dim_{\Z/p} H_1(M_2(L);\Z/p) = r_2(L)$.
If $U$ is an Alexander-trivial subgroup of rank $2d$,
and $B$ is a matrix of $\theta|_U$,
then $\det(B + B^{\top}) = \pm \Alex{\theta|_U}(-1) =  \pm 1$.
So $B + B^{\top}$ has full rank $2d$ over $\Z/p$, and thus $A + A^{\top}$ has rank at least $2d$ over $\Z/p$.
Thus $\coker(A + A^{\top}) \otimes \Z/p \cong H_1(M_2(L);\Z/p)$ has dimension at most $2g - 2d = 2\widetilde{\ga}(\theta)$ over $\Z/p$.
This means that $r_2(L) \leq 2\widetilde{\ga}(\theta)$. Since $r_2$ is invariant under S-equivalence, this
implies (ii).
\end{proof}
\begin{rmk}
Consider a Seifert form $\theta$ of a knot of dimension $2g$
with maximal~$r_2$, i.e.\ $r_2(\theta) = 2g$. Clearly, that condition is equivalent
to the existence of an  odd prime $p$ modulo which $\theta + \theta^{\top}$ vanishes ($\theta + \theta^{\top}$ cannot vanish modulo $2$, since its determinant is odd).
Such a form may be realized as Seifert form of a knot with topological slice genus $g$ (\cref{thm:cg}).
Moreover, all knots $K$ admitting such a Seifert form have some peculiar properties:
for example, the unknotting number of $K$ is bounded below by $\ua(\theta) = 2g$;
and all $2g$ Alexander ideals are non-trivial over $\Z$, since they are sent to $p\Z$ by the substitution $t = -1$.
\end{rmk}
\subsection{The stable algebraic genus}
If $\theta, \zeta$ are Seifert forms with respective Alexander-trivial subgroups $U, V$, then
$\theta\oplus\zeta$ has the Alexander-trivial subgroup $U\oplus V$. This implies that $\ga$ is
subadditive with respect to the connected sum of links:
\[
\ga(L\# L') \leq \ga(L) + \ga(L')
\]
for all links $L$ and $L'$, where the connected sum is taken along arbitrary components. Let us construct an example demonstrating that $\ga$ is in general not additive.
Take $\theta$ to be the Seifert form given by
\[
\begin{pmatrix}
1 & 1 \\ 0 & -1
\end{pmatrix},
\]
for which we have $\ga(\theta) = 1$, since $\Alex{\theta} \neq 1$. On the other hand, the form $\theta\oplus\theta$ admits
an Alexander-trivial subgroup $U$ of rank $2$ generated by $(1,0,0,1)$ and $(0,1,0,0)$. Indeed
\[
\begin{pmatrix}
1 & 0 & 0 & 1\\
0 & 1 & 0 & 0
\end{pmatrix}
\cdot
\begin{pmatrix}
1 & 1  & 0 & 0 \\
0 & -1 & 0 & 0 \\
0 & 0 & 1 & 1  \\
0 & 0 & 0 & -1 \\
\end{pmatrix}
\cdot
\begin{pmatrix}
1 & 0 \\
0 & 1 \\
0 & 0 \\
1 & 0
\end{pmatrix}
=\begin{pmatrix}
0 & 1 \\
0 & -1
\end{pmatrix},
\]
and thus $\ga(\theta \oplus \theta) = 1$.

Following Livingston's definition of the stable slice genus \cite{stable}, one may define the \emph{stable algebraic genus} $\widehat{\ga}(L)$ of a link $L$ as
\[
\widehat{\ga}(L) = \lim_{n\to\infty}\frac{\ga(L^{\# n})}{n}.
\]
Similarly one can define $\widehat{\gtop}(L)$ and $\widehat{\gst}(L)$ for links $L$ with a distinguished component along which the connected sums are taken.
It is now an immediate consequence of \cref{thm:gts<=ga} that for all links $L$ with distinguished component,
\[
\widehat{\ga}(L) \geq \widehat{\gst}(L) \geq \widehat{\gtop}(L).
\]
These inequalities give some motivation for studying the stable algebraic genus.
We will refrain from doing so here, with the exception of the following proposition
which results from strengthening a result of  Baader \cite{baaderIndef} by recasting his argument algebraically, making connections to $r_2$,
and generalizing to multi-component links.

Note that the lower bounds for $2\ga$ in \Cref{prop:bounds}, being additive, are also lower bounds for $2\widehat{\ga}$.
The following proposition shows that taken together they characterize
knots whose stable algebraic genus is strictly less than their genus.
\begin{prop}\label{prop:charstabgalg<g}
For all $r$--component links $L$,
$\widehat{\ga}(L) < g(L)$
holds if and only if
\[
\max\{|\sigma(L)| + \eta(L), r_2(L)\}\ <\ 2g(L) + r - 1.
\]
In particular for all knots $K$, $\widehat{\ga}(K) < g(K)$ holds if and only if
\[
\max\{|\sigma(K)|, r_2(K)\} < 2g(K).
\]
\end{prop}
\begin{proof}
The `only if' part of the statement were discussed in the paragraph preceding the proposition.
Let us now prove the `if' part. For this, fix a Seifert matrix $A$ of $L$ of size $(2g + r - 1) \times (2g + r - 1)$,
where $g$ is the genus of $L$.
Note that $g > 0$. Indeed, $g = 0$ implies $A = A^{\top}$, whence
$A + A^{\top}$ is the zero matrix modulo $2$, which implies $r_2(L) = r - 1$, contradicting the hypotheses.
We denote by $Q_n: \mathbb{Z}^{(2g + r - 1)n} \to \mathbb{Z}$ the quadratic form defined by $A^{\oplus n}$,
i.e.~$Q_n(v) = v^{\top} A^{\oplus n} v$, and let $\mathcal{F}\subset\mathbb{Z}$ be the union of the images of $Q_n$ for all $n \geq 1$.

In a first step, we prove that $\mathcal{F}$ is equal to the subgroup $\mathcal{G}$ of $\Z$ generated by $A_{ij} + A_{ji}$ and $A_{kk}$ for $1\leq i<j\leq 2g+r-1$ and $1\leq k\leq 2g+r-1$.
To show `$\mathcal{F}\subset \mathcal{G}$', we notice that by definition elements $x\in \mathcal{F}$ are sums of elements of the form \[v^{\top} A v=  \sum_{k=1}^{2g+r-1} v_k^2 A_{kk} +
\sum_{1\leq i < j \leq {2g+r-1}} v_iv_j(A_{ij} + A_{ji}),
\] for $v\in\mathbb{Z}^{2g + r - 1}$, which are $\mathbb{Z}$--linear combinations of $A_{ij} + A_{ji}$ and $A_{kk}$.
To show `$\mathcal{F}\supset \mathcal{G}$', it is sufficient to show that $\mathcal{F}$ is a subgroup of $\mathbb{Z}$
and contains all
$A_{ij} + A_{ji}$ and $A_{kk}$. Clearly, $\mathcal{F}$ is non-empty as $0\in \mathcal{F}$.
If $x_1, x_2\in \mathcal{F}$, and $x_{\ell} = Q_{n_{\ell}}(v_{\ell})$, then
\[
x_1 + x_2 = Q_{n_1 + n_2}\begin{pmatrix} v_1 \\ v_2\end{pmatrix}  \in \mathcal{F}.
\]
This proves that $\mathcal{F}$ is closed under addition.
Now, let a non-zero integer $x\in \mathcal{F}$ be given.
We are going to show that $-x\in\mathcal{F}$, which will complete the proof that $\mathcal{F}\subset\mathbb{Z}$ is a subgroup.
Since $|\sigma(L)| + \eta(L) < 2g + r - 1$,
the form $Q_1$ is indefinite, and so there exists $y\in\mathcal{F}$ with the opposite sign of $x$.
We first consider the case $x > 0$ and $y < 0$.
Since $\mathcal{F}$ is closed under addition and thus also under multiplication with positive integers, we find
\[
x \cdot y, (-y - 1)\cdot x \in \mathcal{F} \quad\Rightarrow\quad x \cdot y + (-y - 1)\cdot x = -x \in \mathcal{F},
\]
If instead, we have $x < 0$ and $y > 0$, then we similarly find
\[
(- x) \cdot y, (y - 1)\cdot x \in \mathcal{F} \quad\Rightarrow\quad (-x) \cdot y + (y - 1)\cdot x = -x \in \mathcal{F}.
\]
Finally, let us check that all $A_{ij} + A_{ji}$ and $A_{kk}$ are in $\mathcal{F}$.
Firstly, $Q_1(e_k) = A_{kk}$, where $e_k$ denotes $k$--th standard basis vector.
Secondly,
\[
Q_3
\begin{pmatrix} e_i + e_j \\ -e_i \\ -e_j \end{pmatrix} = A_{ij} + A_{ji}.
\]

As a second step, we prove that $\mathcal{F} = \mathcal{G}=\Z$.
Note that by definition $\mathcal{G}$ contains every entry of $A + A^{\top}$,
i.e.~$A_{ij}+A_{ji}$ and $2A_{kk}$ for $1\leq i<j\leq 2g+r-1$ and $1\leq k \leq 2g+r-1$.
Assume towards a contradiction that $\mathcal{G} \neq \Z$.
This would imply that the greatest common divisor of the entries of $A + A^{\top}$ is non-trivial.
Now recall that $A + A^{\top}$ is a presentation matrix of the first integral homology group of the double branched covering of $L$. Therefore, that homology group would be the sum of $2g + r - 1$ groups of the form $\mathbb{Z}/a_id$ for some $a_i\in\mathbb{Z}$, in contradiction to $r_2(L) < 2g + r - 1$.

To finish, pick two vectors $v_1, v_2\in\mathbb{Z}^{2g + r - 1}$ with $v_1^{\top} (A - A^{\top}) v_2 = 1$.
This is possible since $g > 0$. Since $\mathcal{F}=\Z$, there exist two positive integers $n_1$ and $n_2$ and another two vectors $w_1, w_2$ with $w_i\in\mathbb{Z}^{(2g + r - 1)n_i}$ such that
$Q_{n_1}(w_1) = - v_1^{\top} A v_2$ and $Q_{n_2}(w_2) = v_1^{\top} A (v_2 - v_1)$.
Then
\[
u_1 = \begin{pmatrix} v_1 \\ w_1 \\ w_2 \end{pmatrix} \in \mathbb{Z}^{(2g + r - 1)(1+n_1+n_2)},\quad
u_2 = \begin{pmatrix} v_2 \\ w_1 \\ 0 \end{pmatrix}   \in \mathbb{Z}^{(2g + r - 1)(1+n_1+n_2)}
\]
generate an Alexander-trivial subgroup since
\[
\begin{pmatrix} u_1^{\top} \\ u_2^{\top} \end{pmatrix}
A^{\oplus(1+n_1+n_2)}
\begin{pmatrix} u_1 & u_2 \end{pmatrix} =
\begin{pmatrix} 0 & 0 \\ -1 & (v_2-v_1)^{\top}Av_2 \end{pmatrix}.
\]
This implies that $\ga(L^{\# (n_1 + n_2 + 1)}) < (n_1 + n_2 + 1)g(L) = g(L^{\# (n_1 + n_2 + 1)})$, and therefore $\widehat{\ga}(L) < g(L)$.
\end{proof}
\begin{rmk}
Livingston \cite{stable} constructed a family of knots $K_i$ with \linebreak $\widehat{\gtop}(K_i) < 1$
and $\lim_{i\to\infty} \widehat{\gtop}(K_i) = 1$.
That shows that the hypotheses of \cref{prop:charstabgalg<g} cannot give a lower bound
for the difference $g(K) - \widehat{\ga}(K)$.
\end{rmk}

\section{Three-dimensional characterizations of $\ga$}\label{sec:3d}
\begin{definition}
A \emph{3D-cobordism} between two links $L_1$ and $L_2$ with $r_1$ and $r_2$ components, respectively, is a Seifert surface for a link with $r_1+r_2$ components such that the link given by the first $r_1$ components is $L_1$ and the link given by the other components is $L_2^{\mathrm{rev}}$ with reversed orientation.
\end{definition}
For context, recall that a cobordism between two links $L_0$ and $L_1$ is an oriented connected smooth embedded surface $C$ in $S^3\times[0,1]$ such that $\partial C=L_0\times\{0\}\sqcup L_1^{\mathrm{rev}}\times\{1\}$. 3D-cobordisms correspond to cobordisms $C$ such that the projection $S^3\times[0,1]\to S^3$ to the first factor restricts to an embedding on~$C$. Up to isotopy, this is equivalent to the projection to the second factor being a Morse function all of whose critical points have index~1.

The following proposition provides the proof of \cref{thm:character}.
\begin{prop}\label{prop:3Dcharofga}
For all links $L$, the following holds.
\begin{enumerate}

\item[(i)] Let $F$ be a %
Seifert surface of $L$. Let $K\subset F$ be a simple closed
curve such that $F\setminus K = F_1\sqcup F_2$ with $\partial F_2 = L\sqcup K$
and $\partial F_1 = K$, and $\Alex{K}$ equals 1. Then $\ga(L)$ is the minimal
genus of such a surface $F_2$.
\item[(ii)] The algebraic genus of $L$ equals
the minimum genus among all 3D-cobordisms between $L$ and a knot with Alexander polynomial 1.
\end{enumerate}
\end{prop}
 These should be viewed in light of similar characterizations for the algebraic unknotting number $\ua$ of a knot $K$, which can be defined purely algebraically using the Seifert form, but is most quickly defined as the smallest number of crossing changes needed to turn $K$ into an Alexander polynomial 1 knot.
\begin{proof}
We note that the first statement of \cref{prop:3Dcharofga} is an immediate consequence of \cref{prop:sepcurve}.
Indeed, let $h$ be the smallest genus among surfaces $F_2\subset F_1\sqcup F_2=F\setminus K$ as in (i).
By the definition of $\ga(L)$, there exists a Seifert surface $F$ of some genus $g$ such that
$\ga(L)=g-d$, where $2d$ is the rank of an Alexander-trivial subgroup in $H_1(F;\Z)$; thus, by \cref{prop:sepcurve}, $h\leq \ga(L)$. On the other hand, for any surface $F_1$ as in (i), $H_1(F_1;\Z)\subset H_1(F;\Z)$ is an Alexander-trivial subgroup with respect to the Seifert form on the Seifert surface $F$, thus $h\geq \ga(L)$.

By (i), the second statement of \cref{prop:3Dcharofga} follows, if we establish the following:
given a 3D-cobordism between $L$ and a knot $K$ with Alexander polynomial 1 of some genus $g$, there exists a Seifert surface $F_1$ for $K$ and a 3D-cobordism $F_2$ between $L$ and a knot $K$ of genus $g$ such that $F_1$ and $F_2$ precisely intersect in $K$. This is established in the following Lemma.\end{proof}

\begin{lemma}\label{lemma:separatingCobfromSeifertsurface}
Let two links $L$ and $K$, a Seifert surface $C$ for $K$, and a genus $g$ 3D-cobordism $F'$ between $L$ and $K$ be given. Then $C$ can be stabilized to a Seifert surface $C'$ such that there exists a 3D-cobordism $F$ with genus $g$ between $L$ and $K$ with $C'\cap F=K$.
\end{lemma}
\begin{proof}
The surface $F'$ defines a framing $N_{F'}(K)$ of $K$; i.e.~a disjoint union of embedded annuli, given as a small closed neighborhood of $K$ in $F'$. We first modify $F'$ such that that the induced framing on $K$ agrees with the framing induced by $C$: for every component $K_j$ of $K$ take a properly embedded interval $J_j$ in $F'$ with one boundary point on $K_j$ and the other on $L$, which is possible since $F'$ is connected (by the definition of a Seifert surface).
By inserting full twists along $J_j$ into $F'$ if necessary, we get a new genus $g$ 3D-cobordism $F''$ that induces the correct framing on $K$, which we denote by $N_{F''}(K)$.

Next, we observe that the cobordism $F''$ can be viewed as  arising by adding 1--handles $H_1,\ldots, H_k$ to $N_{F''}(K)$.
More precisely, the following is true.
Let $K'$ be $(\partial N_{F''}(K))\setminus K$; in other words, $K'$ is the parallel copy of $K$ that forms the other part of the boundary of $N_{F''}(K)$. For some non-negative integer $k$, there exist
pairwise disjoint disks $H_1,\ldots, H_k$ in $S^3$ such that
\[F''=N_{F''}(K)\cup H_1\cup\cdots\cup H_k,\] where the $H_i$ are pairwise disjoint disks in $S^3$ such that $H_i\cap N_{F''}(K)$ consists of two closed intervals contained in $K'$. Let $I_i$ denote the core of the handle $H_i$; i.e.~a properly embedded interval in $H_i$ such that its two boundary points lie in the interior of $K'\cap H_i$, one in each component.

Now, we study the intersection between $F''$ and $C$. Since $F''$ and $C$ induce the same framing on $K$, we may isotope them such that $N_{F''}(K)\cap C=K$. We also arrange that the cores $I_i$ intersect $C$ transversely.

\begin{figure}[h]
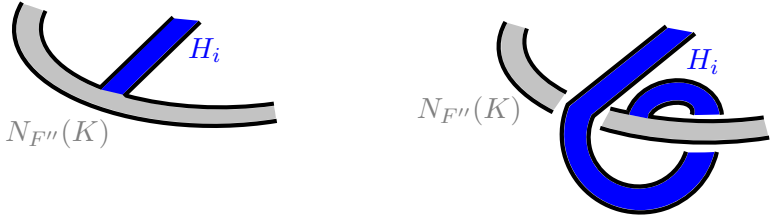
\caption{The Seifert surface $F''$ viewed as the union of the framing $N_{F''}(K)$ (gray) and the 1--handles $H_1,\ldots, H_k$ (blue). The depicted local move (left-to-right) changes the algebraic intersection number between the core of $H_i$ and $C$ by $\pm1$ without changing the genus of $F''$ or the isotopy class of $K$ and $L$.}
\label{fig:modifyF''}
\end{figure}
We now modify $F''$ such that the algebraic intersection number between $I_i$ and $C$ becomes $0$. Indeed, by modifying $F''$ as depicted in \cref{fig:modifyF''},
we change the algebraic intersection number between $I_i$ and $C$ by $\pm 1$. Thus, by modifying $F''$ several times as described in \cref{fig:modifyF''},
we obtain a genus $g$ 3D-cobordism $F$  %
such that the corresponding cores $I_i$ have algebraic intersection number $0$ with $C$.
We note that $F$ is still a genus $g$ 3D-cobordism between $K$ and $L$ since the operation described in \cref{fig:modifyF''} does not change the isotopy type of $L$ or $K$ (however, in general, it does change the isotopy type of $L\cup K$).

In a last step, we show that $C$ can be stabilized such that it no longer intersects $F$. This is done by inductively doing stabilizations on $C$ to reduce the geometric intersection between the cores $I_i$ and $C$ to $0$. Indeed, if $C\cap I_i$ is non-empty, then we find two consecutive (on $I_i$) intersection points $x,y\in C\cap I_i$ of opposite orientation. The subinterval of $I_i$ connecting $x$ and $y$ defines a stabilization of $C$ that intersects $I_i$ in two fewer points than $C$. Inductively, we find a stabilization $C'$ of $C$ which does not intersect any $I_i$ and so, it can be isotoped (rel $K$) away from $F$ except for the intersection at $K$.
\end{proof}

\section{The $\Z$--slice genus}\label{sec:st}
In this section, we establish \cref{thm:gts<=ga}, which states that $\gst(L)\leq\ga(L)$ for all links $L$.
We recall from the introduction:
\begin{definition}
Let the \emph{$\Z$--slice genus} $\gst(L)$ of a link $L$ denote the smallest genus of an oriented connected properly embedded locally flat surface $F$ in the 4--ball $B^4$ with boundary $L\subset S^3$ and $\pi_1(B^4\setminus F)\cong \Z$.
\end{definition}
\begin{proof}[Proof of \cref{thm:gts<=ga}]
Given an $r$--component link $L$, let $F$ be a Seifert surface such that $\ga(L)=g-d$, where $g$ denotes the genus of $F$ and
$2d$ is the rank of an Alexander-trivial subgroup $V$ of $H_1(F;\Z)$.
By \cref{prop:sepcurve}, there exists a separating curve $K$ with Alexander polynomial 1 on $F$ such that $F$ can be written as the following union of surfaces:
\[F=C\cup H_1\cup\cdots\cup H_{2(g-d)+(r-1)};\] where $C$ is a Seifert surface for $K$ of genus $d$ and the $H_i$ are closed disks that are pairwise disjoint and each disk intersects $C$ in two closed intervals that lie in $K=\partial C$. In other words, $F$ is given by attaching $2(g-d)+(r-1)$ many 1--handles to $C$; compare \cref{fig:F}. Compare also with the proof of \cref{lemma:separatingCobfromSeifertsurface}, where we started with a similar setup.
\begin{figure}[h]
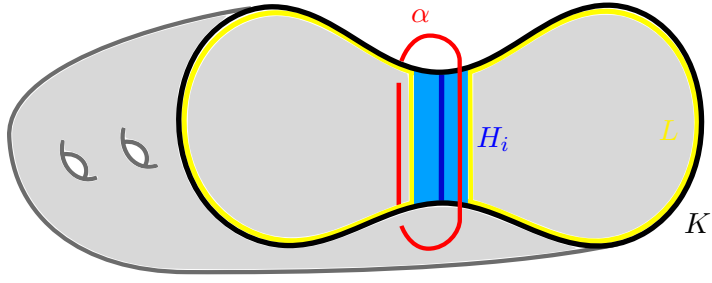
\caption{The Seifert surface $F$ for the link $L$ (yellow) given by attaching 2--dimensional 1--handles $H_i$ (blue) to the Seifert surface $C$ (gray) for the knot $K$ (black). In the case depicted, we have $g=2=
d$ and $r=2$.}
\label{fig:F}
\end{figure}

By \cref{eq:Alex1<=>boundsstddisk}, $K$ bounds a properly embedded locally flat disk $D$ in $B^4$ such that its complement has fundamental group $\Z$.
We may arrange that the disk $D$ meets $S^3$ transversely and is smooth close to $S^3$.

Let $S$ be the following locally flat surface of genus $\ga(L)$ in the $4$--ball $B_2^4$ of radius $2$:
\[S=D\cup H_1\cup\cdots\cup H_{2(g-d)+(r-1)}\cup\Biggl(\,\bigcup_{t\in[1,2]}L_t\Biggr),\]
where $L_t$ denotes the link in the 3--sphere of radius $t$ obtained by stretching $L$ by $t$. %
In particular, $S$ is a witness for $\gtop(L)\leq \ga(L)$; i.e.~we have established \cref{eq:gt<=ga}.
To get the stronger statement $\gst(L)\leq \ga(L)$, it suffices to establish the following claim.

\begin{claim}
The fundamental group of $B_2^4\setminus S$ is isomorphic to $\Z$.
\end{claim}
Briefly said, it turns out that the inclusion $B^4\setminus D\to B_2^4\setminus S$ induces a surjection on $\pi_1$, which implies the claim since $H_1(B_2^4\setminus S;\Z)\cong\Z$ by an appropriate version of Alexander duality. We provide a more detailed argument.

For this, we consider the topological 4--manifold with boundary $B_2^4\setminus N(S)$, where $N(S)$ denotes an open tubular neighborhood of $S$, rather than $B_2^4\setminus S$. The main point is that $B_2^4\setminus N(S)$ (as a topological manifold with boundary) can be obtained
from $B^4\setminus N(D)$ by attaching $2(g-d)+r-1$ many 2--handles: one 4--dimensional 2--handle $\widetilde{H_i}$ corresponding to each $H_i$; compare~\cite[Proposition~6.2.1]{GompfStipsicz}.
In \cref{fig:attachsphereslowdim},
\begin{figure}[t]
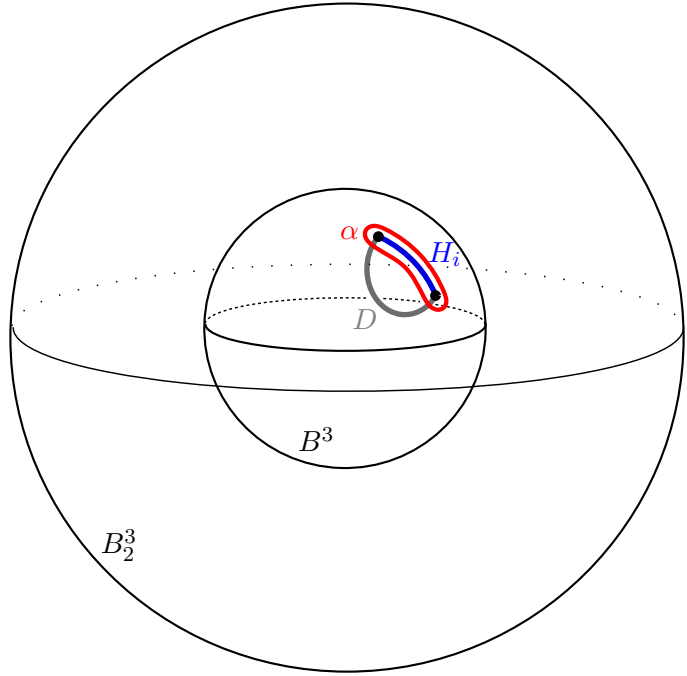
\caption{A knot $S$ in the $3$--ball $B^3_2$ of radius $2$ is given by attaching a 1--handle $H_i\subset S^2$ (blue) to the interval $D\subset B^3$ (gray) along two points (a $0$--dimensional attaching sphere) (black). The knot complement $B^3_2\setminus N(S)$ of $S$ can be obtained by attaching a $3$--dimensional $2$--handle to the solid torus $B^3\setminus N(D)$. The attaching sphere for this $2$--handle is the curve $\alpha\subset S^2$ (red). }
\label{fig:attachsphereslowdim}
\end{figure}
 we illustrate the situation one dimension lower: for a knot in the $3$--ball rather than a surface in the $4$--ball.

We describe the attaching spheres for the handles $\widetilde{H_i}$ in more detail.
Let $I_i$ be a core of the handle $H_i$; i.e.~a properly embedded interval in $H_i$ such that its two boundary points lie in the interior of $C\cap H_i\subset K$, one in each component; compare \cref{fig:F}.
Choose closed disks $D_i$ in $S^3$ such that each $D_i$ intersects $\overline{F\setminus C}=\bigcup_{j=1}^{2(g-d)+r-1}H_i$ only in the interior of $D_i$, and such that the intersection is $I_i$. Let $\alpha_i$ be the boundary curve of $D_i$; i.e.~a curve that wraps once `around' $H_i$ while staying close to $I_i$; see \cref{fig:F}. We leave it to the reader to check that, indeed, $B_2^4\setminus N(S)$ (as a topological manifold with boundary) is obtained
from $B^4\setminus N(D)$ by attaching $2(g-d)+r-1$ many $2$--handles, one along each $\alpha_i$; compare~\cite[Proof of Proposition~6.2.1]{GompfStipsicz}.

In particular, we have that $B_2^4\setminus N(S)$ deformation retracts to the topological space $X$ obtained by gluing $2(g-d)+r-1$ many disks along $\alpha_i$ to $B^4\setminus N(D)$. Therefore, we have
\[
\pi_1(B_2^4\setminus N(S))
\cong \pi_1(X)\cong \frac{\pi_1(B^4\setminus N(D))}{\langle[\alpha_1],\ldots,[\alpha_{2(g-d)+(r-1)}]\rangle}
\]
by the Seifert-van Kampen Theorem, where $[\alpha_i]$ denotes the homotopy class of $\alpha_i$ in $\pi_1(B^4\setminus N(D))$ (a base point may be appropriately chosen). However, note that $\alpha_i$ is null-homologous in $S^3\setminus K$, i.e.~the algebraic linking number of $\alpha_i$ and $K$ is zero; since $\alpha_i$ is clearly homologous to a meridian of $K$ plus an oppositely oriented meridian of $K$. In particular, the $\alpha_i$ are also null-homologous in $B^4\setminus N(D)$. Since $\pi_1(B^4\setminus D)\cong \Z$, we have $\pi_1(B^4\setminus N(D))\cong H_1(B^4\setminus N(D); \Z)$, and so
the $\alpha_i$ are also null-homotopic in $B^4\setminus N(D)$. With this we conclude
\[
\pushQED{\qed}
\pi_1(B_2^4\setminus N(S))
\cong \frac{\pi_1(B^4\setminus N(D))}{\langle[\alpha_1],\ldots,[\alpha_{2(g-d)+(r-1)}]\rangle}
\cong \pi_1(B^4\setminus N(D))\cong\Z.
\qedhere
\popQED
\]
\renewcommand{\qedsymbol}{}
\end{proof}

\section{Algebraic genus and algebraic unknotting number}\label{sec:ua}
In this section, we relate the algebraic genus and the algebraic unknotting number of knots as follows:
\uathm*
We consider knots rather than all links since, a priori, the invariant $\ua$ is only a knot invariant (rather than an invariant of links) and, for now, we do not know of a generalization of $\ua$ to links such that \cref{thm:galequaleq2ga} holds.

We prove the two inequalities of \cref{thm:galequaleq2ga} using different interpretations of $\ua$.
Using that $\ua(K)$ is equal to the minimum number of crossing changes needed to transform $K$ into a knot with Alexander polynomial $1$ \cite{saeki,fogel},
the first inequality $\ga\leq\ua$ is an immediate consequence of the following proposition (which might be of independent interest):
\begin{prop}
Let $L_1, L_2$ be two links related by a crossing change. Then $|\ga(L_1) - \ga(L_2)| \leq 1$.
\end{prop}
\begin{proof}
Applying Seifert's algorithm to diagrams of $L_1$ and $L_2$ that differ by one crossing change, one finds two Seifert surfaces, say of genus $g$. A good choice of basis for the first homology of these Seifert surfaces yields $2g\times 2g$
Seifert matrices $M_i$ for $L_i$ such that $M_2=M_1\pm e_{11}$, where $e_{11}$ denotes the square $2g\times 2g$ matrix with top-left entry $1$ and all other entries equal to zero.
By stabilizing $M_1$, we may assume that $\ga(L_1)=\widetilde{\ga}(M_1)$; i.e.~the maximal Alexander-trivial subgroup of $\mathbb{Z}^{2g}$ with respect to $M_1$ is of rank $2g-2\ga(L_1)$.
Let us apply the same stabilizations to $M_2$, so that the property $M_2 = M_1 \pm e_{11}$ is retained.

Now consider the following $(2g+2)\times (2g+2)$ matrix obtained as a stabilization of $M_2$:
\begin{equation*}\widetilde{M_2}=\left(
{\begin{tabular}{ccc|cc}&&&$\mp 1$&0\\
&$M_{2}$&&0&0\\
&&&$\vdots$&$\vdots$\\
&&&$0$&$0$\\
\hline
0&$\cdots$&0&0&1%
\\
0&$\cdots$&0&0%
&0\\
\end{tabular}}
\right).\end{equation*}
A change of basis turns $\widetilde{M_2}$ into
\begin{equation*}\left(
{\begin{tabular}{ccc|cc}&&&$\mp 1$&1\\
&$M_2\mp e_{11}$&&0&0\\
&&&$\vdots$&$\vdots$\\
&&&$0$&$0$\\
\hline
0&$\cdots$&0&0&1%
\\
0&$\cdots$&0&0%
&0\\
\end{tabular}}
\right)=\left(
{\begin{tabular}{ccc|cc}&&&$\mp 1$&1\\
&$M_1$&&0&0\\
&&&$\vdots$&$\vdots$\\
&&&$0$&$0$\\
\hline
0&$\cdots$&0&0&1%
\\
0&$\cdots$&0&0%
&0\\
\end{tabular}}
\right);\end{equation*}
indeed, the latter is obtained from $\widetilde{M_2}$ by adding the second-to-last column to the first column and, correspondingly, adding the second-to-last row to the first row.
Since there is an Alexander-trivial subgroup of rank $2g-2\ga(L_1)$ with respect to $M_1$, the same holds for $\widetilde{M_2}$. Consequently,
\[
\ga(L_2)\leq \widetilde{\ga}(\widetilde{M_2}) \leq g+1-(g-\ga(L_1))=\ga(L_1)+1,
\]
where the two inequalities are immediate from the definition of $\ga$.
This gives $\ga(L_2) - \ga(L_1) \leq 1$, and by switching the roles of $L_1$ and $L_2$
also  $\ga(L_1) - \ga(L_2) \leq 1$, which concludes the proof.
\end{proof}
We point out that the stabilizations in the first paragraph in the above proof are necessary as it remains an open question whether $\widetilde{\ga}(\theta)=\widetilde{\ga}(\widetilde{\theta})$ holds for all S-equivalent Seifert forms $\theta$ and $\widetilde{\theta}$.

To tackle the second inequality of \cref{thm:galequaleq2ga}, we use Friedl and Borodzik's knot invariant $n$, which they show to be equal to $\ua$~\cite{BorodzikFriedl_15_TheUnknottingnumberAndClassInv1,BorodzikFriedl_14_OnTheAlgUnknottingNr}.
Let us briefly give the necessary definitions.
Let $\Lambda = \Z[t^{\pm 1}]$ be a ring with involution $a \mapsto \overline{a}$ given by the linear extension of $t \mapsto t^{-1}$,
and let $\Omega$ be its quotient field. For a Hermitian $m\times m$ matrix $A$ over $\Lambda$ that is invertible over $\Omega$, denote by $\lambda(A)$ the Hermitian form
\[
\Lambda^m / {A \Lambda^m} \times
\Lambda^m /{A \Lambda^m} \to \Omega/{\Lambda},\qquad
(a, b) \mapsto \overline{a}^{\top} A^{-1} b.
\]
Suppose $V$ is a $2g\times 2g$ Seifert matrix of a knot $K$ of the following kind:
\begin{equation}\label{eq:1}
\begin{pmatrix}
B & C + \ID \\
C^{\top} & D
\end{pmatrix},\qquad
\begin{matrix}
B, C, D \text{ are } g\times g,\\
B, D \text{ are symmetric}.
\end{matrix}
\end{equation}
Note that any Seifert matrix of $K$ is congruent to one of this kind.
Then the {Blanchfield pairing} is isometric to $\lambda(\widetilde{V})$, where $\widetilde{V}$ is the
following Hermitian matrix over $\Lambda$ (see \cite{Ko} and formulas (2.3) and (2.4) in \cite{BorodzikFriedl_15_TheUnknottingnumberAndClassInv1}):
\[
\begin{pmatrix}
B & -t\ID + (1-t)C \\
-t^{-1}\ID + (1-t^{-1})C^{\top} & x\cdot D
\end{pmatrix},
\]
where we use the shorthand $x = (1 - t) + (1 - t^{-1}) = (1 - t) \cdot (1 - t^{-1})$.
For a knot $K$, $n(K)$ is defined as the minimal size of a Hermitian matrix $A$ over $\Lambda$
such that $\lambda(A)$ is isometric to the Blanchfield pairing of $K$,
and the integral matrix $A(1)$ is congruent to a diagonal matrix.

We will use the following classical result on integral forms; see e.g.~\cite{huse}.
\begin{lemma}\label{thm:1}
An indefinite odd unimodular symmetric integral form can be represented by a diagonal matrix.\qed
\end{lemma}
Here, an integral form $\theta$ is called even if $\theta(v,v)\in 2\Z$ for all $v$,
and odd otherwise. Note the sum of two even forms is even, and the sum of an even and an odd form is odd.
\begin{lemma}\label{lem:1}
Any matrix of the kind \cref{eq:1} is congruent to another matrix of the kind \cref{eq:1} with $B$ representing an odd form.
\end{lemma}
\begin{proof}
Suppose $B$ is even. Then we distinguish two cases, depending on whether $D$ is even as well. If it is, a simple change of basis yields the following
congruent matrix:
\[
\begin{pmatrix}
B + C + C^{\top} + \ID + D & C + D + \ID \\
C^{\top} + D & D
\end{pmatrix},
\]
and $B + C + C^{\top} + \ID + D$ is odd, because $B$, $D$, and $C+C^{\top}$ are even,
while $\ID$ is odd.
If, on the other hand, $D$ is odd, then
again, a simple change of basis gives the congruent matrix:
\[
\pushQED{\qed}
\begin{pmatrix}
D & -C^{\top} \\
- C - \ID & B
\end{pmatrix}.
\qedhere
\popQED
\]
\renewcommand{\qedsymbol}{}
\end{proof}
Before addressing $\ua\leq 2\ga$, let us warm up by directly proving \cref{cor:ualeqdegAlex}
(which also follows from \cref{thm:galequaleq2ga}); i.e.~we show that for all knots $K$, we have
\[\ua(K)=n(K) \leq \deg \Alex{K}.\]
\begin{proof}[Proof of \cref{cor:ualeqdegAlex}]
It is well-known that the S-equivalence class of Seifert forms of $K$ contains a Seifert matrix $V$ of size $\deg\Alex{K}$;
indeed, every S-equivalence class has a non-singular representative \cite{Trotter_62_HomologywithApptoKnotTheory},
whose dimension must equal the degree of the Alexander polynomial.

After a basis transformation, we may assume that $V$ is of the kind \cref{eq:1}.
By \cref{lem:1} we may assume $B$ represents an odd form. Thus
\[
\widetilde{V}(1) = \begin{pmatrix}
B & -\ID \\
-\ID & 0
\end{pmatrix}
\]
is odd as well; and furthermore symmetric, unimodular and indefinite, and thus congruent over $\Z$
to a diagonal matrix by \cref{thm:1}. Since $\lambda(\widetilde{V})$ represents the Blanchfield pairing, this concludes the proof.
\end{proof}
\begin{proof}[Proof of the second inequality of \cref{thm:galequaleq2ga}]
Let $V$ be a Seifert matrix of $K$ of size $2g$ with an Alexander-trivial subgroup of rank $2g-2\ga(K)$. Such a $V$ exists by the definition of $\ga$.
Using \cref{lem:stdformofalextriv}, one may change the basis such that the first $2g-2\ga(K)$ basis vectors generate the Alexander-trivial subgroup, and such that $V$ appears as follows:
\[
\raisebox{.3ex}{$%
\begin{matrix}
\scriptstyle \tiny g-\ga(K)\\
\scriptstyle \tiny g-\ga(K)\\
\scriptstyle \tiny \ga(K)\\
\scriptstyle \tiny \ga(K)\\
\end{matrix}$}\quad
\begin{pmatrix}
0        & \ID + U     & E & F        \\
U^{\top} & 0           & G     & H \\
E^{\top} & G^{\top}    & B            & \ID + C \\
F^{\top} & H^{\top}    & C^{\top}     & D
\end{pmatrix}.
\]
Here, all upper-case letters denote square matrices, whose sizes are indicated to the left of the matrix.
The matrix $U$ is upper triangular with zeros on the diagonal;
and the matrices $B$ and $D$ are symmetric. Furthermore, we may assume $B$ represents an odd form
by applying \cref{lem:1} to the lower right $2\ga(K) \times 2\ga(K)$ submatrix of $V$;
note that the involved basis change does not affect the upper left quadratic submatrix of size $2(g-\ga(K))$.

If one swaps the second and third column and second and third row of $V$, one obtains a matrix $V'$ of the kind \cref{eq:1}, so $\lambda$
of the following matrix $\widetilde{V'}$ is isometric to the Blanchfield pairing:
\[
\footnotesize
\begin{pmatrix}
0        & E        & -t\ID + (1-t)U & (1-t)F \\
E^{\top} & B        & (1-t)G^{\top}               & -t\ID + (1-t)C \\
-t^{-1}\ID + (1-t^{-1})U^{\top} & (1-t^{-1})G & 0               & xH \\
(1-t^{-1})F^{\top} & -t^{-1}\ID + (1-t^{-1})C^{\top} & xH^{\top}         & xD
\end{pmatrix}.
\]
Again swapping the second and third column and row now gives a matrix $W_1$ equal to
\[
\footnotesize
\begin{pmatrix}
0        & -t\ID + (1-t)U & E & (1-t)F \\
-t^{-1}\ID + (1-t^{-1})U^{\top} & 0        & (1-t^{-1})G & xH \\
E^{\top} & (1-t)G^{\top} & B               & -t\ID + (1-t)C \\
(1-t^{-1})F^{\top} &xH^{\top} & -t^{-1}\ID + (1-t^{-1})C^{\top} & xD
\end{pmatrix}.
\]
Since $\det(-t\ID + (1-t)U) = (-t)^{\ga}$ is a unit in $\Lambda$, there is an inverse $S = (-t\ID + (1-t)U)^{-1}$ over $\Lambda$.
Now let the transformation matrix $T$ be
\[
\begin{pmatrix}
\ID & 0 & -(1-t^{-1})\overline{S}^{\top}G  & -x\overline{S}^{\top}H \\
0 & \ID & -SE & -(1-t)SF \\
0 & 0 & \ID & 0 \\
0 & 0 & 0 & \ID \\
\end{pmatrix}.
\]
Note that $\det(T) = 1$. One may compute $W_2 = \overline{T}^{\top} W_1 T$ to be the block sum of
the quadratic matrix of size $2(g - \ga(K))$
\[
W_3 = \begin{pmatrix}
0        & -t\ID + (1-t)U \\
-t^{-1}\ID + (1-t^{-1})U^{\top} & 0
\end{pmatrix}\]
and another quadratic matrix $W_4$ of size $2\ga(K)$, which we do not write out for aesthetic reasons.
The first block $W_3$ can be split off because it has determinant $1$.
In other words, since $\lambda(W_2)$ is isometric to the Blanchfield
pairing, and $W_2 = W_3 \oplus W_4$, we find $\lambda(W_4)$ to be isometric to the Blanchfield pairing as well.
If $W_4$ evaluates at $t = 1$ to an integral matrix
that is congruent to a diagonal matrix, then we have proven that
$\ua(K) = n(K) \leq 2\ga(K)$, as desired.
Note that
\[
W_1(1) = \begin{pmatrix}
0 & -\ID & E(1) & 0 \\
-\ID & 0 & 0 & 0 \\
E(1) & 0 & B(1) & -\ID \\
0 & 0  & -\ID & 0
\end{pmatrix}
\]
and
\[
T(1) =
\begin{pmatrix}
\ID & 0 & 0 & 0 \\
0 & \ID & E(1) & 0 \\
0 & 0 & \ID & 0 \\
0 & 0 & 0 & \ID \\
\end{pmatrix}\qquad\Rightarrow\qquad
W_4(1) =
\begin{pmatrix}
B(1) & -\ID \\
-\ID & 0
\end{pmatrix}.
\]
So $W_4(1)$ is indeed congruent to a diagonal matrix by \cref{thm:1}.
\end{proof}
\begin{Example}
There is no obvious way in which \cref{thm:galequaleq2ga} could be sharpened, since each of the two inequalities in that theorem
 may be an equality.
One need not look far for examples.
On the one hand, $\ga(K) = \ua(K)$ occurs e.g. for knots $K$ with $|\sigma(K)| = 2g(K)$, such as 2--stranded torus knots.
On the other hand, the algebraic unknotting number may exceed the 3--genus of a knot, which is an upper bound for the
algebraic genus; e.g.~$\ua(7_4) = 2, g(7_4) = \ga(7_4) = 1$, or $\ua(9_{49}) = 3, g(9_{49}) = \ga(9_{49}) = |\sigma(9_{49})|/2 = 2$.
\end{Example}

\section{The algebraic genus of fibered knots}\label{sec:fibred}

By the definitions (compare \cref{sec:intro} and \cref{def:galg}), we have $t(K)\leq\ga(K)$ for all knots, where $t$ denotes Taylor's invariant. A priori, $\ga$ can be arbitrarily larger than Taylor's invariant $t$. In particular, for knots $K$ with Alexander polynomial of degree $4$, $\ga(K)\leq 2=\frac{\deg\Alex{K}(t)}{2}$; and one would suspect that $2$ can be attained independently of the value of $t(K)$. However, it turns out that if additionally $t(K)=0$, i.e.~$K$ is algebraically slice, and $\Alex{K}(t)$ is monic, then $\ga(K)$ is at most 1. In fact, we show the following.
\begin{prop}\label{prop:fibred}
  If a knot $K$ is algebraically slice and has monic Alexander polynomial of degree $4$, then
$\ga(K)=1$.
\end{prop}
To the authors this was surprising; for example, since fibered knots are known to have monic Alexander polynomial, this yields the following.
\begin{corollary}
Algebraically slice, genus $2$, fibered knots have topological slice genus at most $1$.\qed
\end{corollary}
By recalling the definition of an algebraically slice knot (one, and thus all, Seifert matrices are metabolic) and the fact that every knot has a Seifert matrix of size $\deg\Alex{K}(t)\times \deg\Alex{K}(t)$ up to $S$-equivalence (which is implied by the fact that all Seifert matrices are $S$-equivalent to one with non-zero determinant as proven by Trotter~\cite{Trotter_62_HomologywithApptoKnotTheory}), \cref{prop:fibred} follows from the following lemma.
\begin{lemma}\label{prop:4x4}
Let $A$ be an integral metabolic $4\times 4$ matrix.
Suppose $A$ and $A-A^{\top}$ are invertible.
Then there is a subgroup of $\Z^4$ of rank two restricted to which
$A$ has the form
\[
\begin{pmatrix}
0 & 1 \\
0 & *
\end{pmatrix}
\]
\end{lemma}
Before providing the proof, which consists of an elementary calculation, we provide an example that shows that no analog statement holds for $6\times 6$ Seifert matrices.
\begin{Example}\label{ex:r2}
Let $K$ be a knot with the following Seifert matrix of det 1:
\[M=\begin{pmatrix}
0 & 0 & 0 & 1 & 0 & 0\\
0 & 0 & 0 & 0 & 1 & 0\\
0 & 0 & 0 & 0 & 0 & 1\\
-4 & -3 & -6 & 0 & 0 & 0\\
-3 & -1 & -3 & 0 & 0 & 0\\
-6 & -3 & -7 & 0 & 0 & 0\\
\end{pmatrix}.
\]
By definition $K$ is algebraically slice; however, $M + M^{\top}$ is the zero matrix modulo $3$, whence $\ga(K)=3=\frac{\deg\Alex{K}(t)}{2}$ by \Cref{prop:bounds}(ii).
\end{Example}
\begin{proof}[Proof of \cref{prop:4x4}]
By assumption, $A$ is of the form
\[
\left(
\begin{array}{c|c}
0 & U \\\hline
V & *
\end{array}\right)
\]
for $2\times 2$ matrices $U, V$.
Invertibility of $A$ is inherited by $U$ and $V$, so $A$ is congruent to
\begin{equation}\label{eq:small1}
\left(
\begin{array}{c|c}
\ID & 0 \\\hline
\rule{0pt}{2.1ex}0 & U^{-t}
\end{array}\right)\cdot
\left(
\begin{array}{c|c}
0 & U \\\hline
\rule{0pt}{2.1ex}V & *
\end{array}\right)
\cdot
\left(
\begin{array}{c|c}
\ID & 0 \\\hline
\rule{0pt}{2.1ex}0 & U^{-1}
\end{array}\right)
 =
\left(
\begin{array}{c|c}
0 & \ID \\\hline
\rule{0pt}{2.1ex}V' & *
\end{array}\right).
\end{equation}
Let
\[
V' = \begin{pmatrix}
a & b \\
c & d
\end{pmatrix}.
\]
If $a = 0$, the subgroup generated by the first and third basis vector is of the desired form
with respect to the matrix \cref{eq:small1}, and similarly if $d = 0$.
For any invertible $2\times 2$ matrix $T$, the matrix \cref{eq:small1} is congruent to:
\[
\left(
\begin{array}{c|c}
T^{\top} & 0 \\\hline
\rule{0pt}{2.2ex}0 & T^{-1}
\end{array}\right)\cdot
\left(
\begin{array}{c|c}
0 & \ID \\\hline
\rule{0pt}{2.2ex}V' & *
\end{array}\right)
\cdot
\left(
\begin{array}{c|c}
T & 0 \\\hline
\rule{0pt}{2.2ex}0 & T^{-\top}
\end{array}\right)
 =
\left(
\begin{array}{c|c}
0 & \ID \\\hline
\rule{0pt}{2.2ex}T^{-1}V'T & *
\end{array}\right).
\]
So one may try to decrease $|a|$ by replacing $V'$ by $T^{-1}V'T$.
Assume this is no longer possible.
Taking
\[
T = \begin{pmatrix}
1 & \pm 1 \\
0 & 1
\end{pmatrix}
\]
replaces $a$ by $a \mp c$. So our assumption yields that $|a| \leq |c|/2$,
and similarly $|a| \leq |b|/2$.
Taking
\[
T = \begin{pmatrix}
0 & 1 \\
-1 & 0
\end{pmatrix}
\]
switches $a$ and $d$, so we also have $|a| \leq |d|$.
We have $\det V' = \pm 1$, and $\det (V'-\ID) = \pm 1$.
Thus
\begin{align*}
|\det V' - \det (V'-\ID)| & \leq 2 \Rightarrow \\
|(ad-bc) - (ad - bc - a - d + 1)|  & \leq 2 \Rightarrow  \\
|a + d| & \leq 3.
\end{align*}
Therefore,
\begin{align*}
bc & = ad \pm 1 \Rightarrow \\
|bc| & \leq |a||(a+d) - a| + 1 \\
 & \leq |a|^2 + |a+d||a| + 1 \\
 & \leq |bc|/4 + 3|c|/2 + 1 \Rightarrow \\
0  & \leq -3|bc| + 6|c| + 4 \Rightarrow \\
0  & \leq (6-3|b|)|c| + 4.
\end{align*}
This implies that if $|b| \geq 3$, then $|c| \leq 1$, which in turn implies $|a| = 0$ by our assumption.
Since one may switch the role of $b$ and $c$, the only remaining case is
$|b|, |c| = 2$ and $|a|, |d| = 1$. But these values contradict that $\det V' = 1$.
\end{proof}

\section{On the optimality of slice genus bounds}\label{sec:optimality}
The algebraic genus $\ga(L)$ of a link $L$ is an upper bound for the topological slice genus of $L$ that
depends only on the S-equivalence class of Seifert forms for $L$.
\Cref{q:galg=bestclassicalupperbound} asks if it is the best bound with that property. In this section, we pursue this and
related questions on the optimality of slice genus bounds.
To make the dependency on the Seifert form more precise, let us fix for each Seifert form $\theta$
the following sets of links:\\
\begin{align*}
\mathcal{E}_{\theta} & = \Bigl\{\parbox{75mm}{\centering links that admit a Seifert surface with\linebreak Seifert form isometric to $\theta$}\Bigr\}, \\
\rotatebox{90}{$\supset$} \\
\mathcal{S}_{\theta} & = \Bigl\{\parbox{75mm}{\centering links that admit a Seifert surface with\linebreak Seifert form S-equivalent to $\theta$}\Bigr\}, \\
\rotatebox{90}{$\supset$} \\
\mathcal{C}_{\theta} & = \Bigl\{\parbox{75mm}{\centering links that admit a Seifert surface with\linebreak Seifert form algebraically concordant to $\theta$}\Bigr\}. \\
\end{align*}
Written in this notation, the statement of \Cref{prop:gaS-equivalenceclass} is that
\[
\forall\, L, L' \in \mathcal{S}_{\theta} \quad\Rightarrow\quad \ga(L) = \ga(L'),
\]
the inequality \Cref{eq:gt<=ga} says that for all Seifert forms $\theta$,
\begin{equation}\label{eq:gtop<=ga}
\max_{L'\in\mathcal{S}_{\theta}} \gtop(L') \leq  \ga(\theta),
\end{equation}
and \Cref{q:galg=bestclassicalupperbound} asks whether the inequality \cref{eq:gtop<=ga} is in fact an equality.
This question is about the slice genus. A corresponding qualitative question about sliceness
would be: which Seifert forms guarantee the sliceness of a link? More specifically, given a Seifert form $\theta$, does the following hold:
\begin{equation}\label{eq:opt}
\ga(\theta) > 0 \quad\Rightarrow\quad \max_{L\in\mathcal{E}_{\theta}} \gtop(L) > 0?
\end{equation}

The rest of this section pursues \cref{q:galg=bestclassicalupperbound} for knots. Except for \cref{rmk:opt}, where answers for related questions are provided, we only consider knots rather than links in the rest of this section.

Note that for the Seifert form of a knot, $\ga(\theta) = 0 \Leftrightarrow \Alex{\theta} = 1$.
Livingston~\cite{livingstonseifert} proved \Cref{eq:opt} for all Seifert forms $\theta$ of knots satisfying a technical condition on the Alexander polynomial of $\theta$.
For this, he used Casson-Gordon obstructions to sliceness, which involve the $d$--fold branched covers of a knot for prime powers $d$.
For non-prime powers, no such obstructions are available, which is precisely the reason that Casson-Gordon obstructions do not solve \cref{eq:opt} for all Seifert forms of knots.
The proof of \Cref{eq:opt} for Seifert forms of knots was completed by Kim using $L^2$--invariants---for Seifert forms of links,
it appears to be open.
\begin{thm}[\cite{kim}]\label{thm:kim}
Let $\theta$ be the Seifert form of a knot with Alexander polynomial $\Alex{\theta}$ not equal to $1$.
Then $\theta$ is realized as the Seifert form of a knot that is not topologically slice.
\end{thm}
Now, the strategy to attack the quantitative question must be
to construct for a given Seifert form $\theta$ a knot $K$ realizing $\theta$ with $\gtop(K) = \ga(K)$;
or, to obtain partial results, with $\gtop(K)$ as high as possible. To this end, one needs
lower bounds for the topological slice genus. To the best knowledge of the authors, there
are only three such bounds (disregarding those which are in fact only obstructions to sliceness):
\begin{itemize}
\item the Seifert form bounds (such as Levine-Tristram signatures),
      subsumed by the bound coming from Taylor's invariant $t(\theta)$ \cite{taylor};
\item the bounds from Casson-Gordon invariants \cite{gilmer};
\item and bounds coming from $L^2$--signatures \cite{Cha_08_TopMinGenusandL2}.
\end{itemize}
In what follows, we will apply the first two bounds of that list to the problem, and obtain partial results.
A fully affirmative answer would in all probability require stronger lower bounds for the topological slice genus than the ones at our disposal.

Since Taylor's bound is determined by the Seifert form, its only contribution to \cref{q:galg=bestclassicalupperbound} is
\[
t(\theta) \leq \max_{K\in\mathcal{S}_{\theta}} \gtop(K)
\]
for Seifert forms $\theta$ of knots.
In fact, Taylor's bound is rather the solution to the opposite problem---it is the optimal \emph{lower} bound determined by
the Seifert form.
The corresponding question for Seifert forms coming from links with more than one component appears to be open.

Let us now apply Casson-Gordon obstructions to \cref{q:galg=bestclassicalupperbound}.
We briefly fix our notations for branched covers and recall some of their well-known properties (cf.~e.g.~\cite{rolfsen_knotsandlinks}).
Let $K$ be a knot with Seifert form $\theta$. For a positive integer $d$, we write $M_d(K)$ for the $d$--fold branched covering of $S^3$ along $K$.
The homology group $H_1(M_d(K); \Z)$ is one of the oldest knot invariants. If $d$ is a prime power, then
$H_1(M_d(K);\Z)$ is a finite group. If $d$ is an odd prime power, that group is equal to $G\oplus G$ for some group $G$.
Denote by $r_d(K)$ the minimum number of generators of $H_1(M_d(K);\Z)$. Then
\begin{equation}\label{eq:rd}
0 \leq r_d(K) \leq \deg(\Alex{K}(t)),
\end{equation}
and $r_d(K)$ is even if $d$ is odd. While the order of $H_1(M_d(K);\Z)$ is determined by $\Alex{K}$, this is not the case for $r_d(K)$, which is,
however, determined by the S-equivalence class of $\theta$: e.g.~if $M$ is a matrix for $\theta$, and $P = (M^{\top} - M)^{-1} M^{\top}$, then
$P^d - (P - \ID)^d$ is a presentation matrix of $H_1(M_d(K);\Z)$. Thus, we also write $r_d(\theta)$ instead of $r_d(K)$.
Our result is now the following:
\begin{prop}\label{thm:cg}
Every Seifert form $\theta$ of a knot is realized by a knot $K$ with
\[
\phantom{\max_{K\in\mathcal{S}_{\theta}}}
\gtop(K) \geq \max_{\substack{d\text{ {prime}}\\\text{{power}}}} \biggl\lceil \frac{r_d(\theta)}{2(d - 1)} \biggr\rceil.
\]
\end{prop}
\noindent Of course, this implies
\[
\max_{K'\in\mathcal{S}_{\theta}} \gtop(K') \geq
\max_{\substack{d\text{ {prime}}\\\text{{power}}}} \biggl\lceil \frac{r_d(\theta)}{2(d - 1)} \biggr\rceil.
\]
It also recovers the inequality
$2\ga(L) \geq r_2(L)$,
which was proved earlier on in \cref{prop:bounds}(ii).

Before proving \Cref{thm:cg}, we fix our notation of Casson-Gordon invariants.
The first integral homology group of $M_d(K)$ admits a linking form $\beta_d: M_d(K) \times M_d(K) \to \Q/\Z$, which is determined by $\theta$.
We write $H_1(M_d(K);\Z)^*$ for the group of characters $\chi$ of $H_1(M_d(K))$, i.e.~homomorphisms $H_1(M_d(K); \Z) \to \Q/\Z$.
The linking form then induces a  dual form $\beta_d^*$ on $H_1(M_d(K);\Z)^*$.
Casson and Gordon associate to $(K, d, \chi)$ an invariant
\[
\tau(K, d, \chi) \in W(\C(t))\otimes \Q,
\]
where $W(\,\cdot\,)$ denotes the Witt group~\cite{CassonGordon_86,cg}. One may take signatures of $\tau(K, d, \chi)$, who obstruct
the topological sliceness of $K$. We will only need the ordinary signature, which we denote by $\sigma\tau(K, d, \chi)$.
Gilmer showed how this knot invariant induces lower bounds for the topological slice genus:
\begin{lemma}[\cite{gilmer}]\label{gilmer}
For a knot $K$ and a prime power $d$, the form $\beta^*$ on $H_1(M_d(K);\Z)^*$ decomposes as a direct sum of forms $\beta_1$ and $\beta_2$
on $G_1$ and $G_2$, respectively; such that $G_1$ may be generated by
$2(d-1)\gtop(K)$ elements, and $\beta_2$ admits a metabolizer $H$ in which all non-trivial characters $\chi\in H$ of prime power order satisfy
\begin{equation}\label{eq:cgineq}
\pushQED{\qed}
\biggl|\sigma\tau(K, d, \chi) + \sum_{j=1}^{d-1} \sigma_{j/d}(K)\biggr| \leq 2d\gtop(K).
\qedhere
\popQED
\end{equation}%
\end{lemma}
Note that Gilmer proved this statement in the smooth category, but it is known to carry over to the topological category by the
the work of Freedman.

The proof of \cref{thm:cg} strongly relies on the techniques Livingston used for proving his partial resolution of \cref{eq:opt}.
We will need the following construction:
\begin{lemma}[{\cite[Theorem 4.6]{livingstonseifertv2}}]\label{lem:livrel}
For every Seifert form $\theta$ of a knot, every positive number $C$, and every prime power $d$,
there is a knot $K$ realizing $\theta$ with $\sigma\tau(K, d, \chi) > C$ for all non-trivial characters $\chi\in H_1(M_d(K);\Z)^*$.\hfill\qed
\end{lemma}
The knot $K$ in \cref{lem:livrel} is constructed from an arbitrary knot $K'$ that realizes $\theta$
by \emph{infection}, i.e.~satellite operations that do not change the Seifert form, but affect the Casson-Gordon invariants
to an extent dependent on the signatures of the pattern knots (as was determined by Litherland~\cite{Litherland_79_SignaturesOFIteratedTorusKnots}).
\begin{proof}[Proof of \cref{thm:cg}]
Let a prime power $d$ be fixed, and set
\[
C = d\cdot \dim \theta + \sum_{j=1}^{d-1} \sigma_{j/d}(K).
\]
By \cref{lem:livrel}, there is a knot $K$ realizing $\theta$ with $\sigma\tau(K, d, \chi) > C$ for all non-trivial $\chi\in H_1(M_d(K);\Z)^*$.
Since $2\gtop(K) \leq \dim\theta$, no non-trivial character $\chi$ satisfies \cref{eq:cgineq}. Therefore, in the notation of \cref{gilmer},
$G_2$ is trivial, and thus $H_1(M_d(K);\Z)^* = G_1$ cannot be generated by less than $2(d-1)\gtop(K)$ elements. The statement of \Cref{thm:cg} follows.
\end{proof}
\begin{rmk}
Let us discuss how to apply \cref{thm:cg}. First of all, to determine
\begin{equation} \label{maxrd}
\max_{\substack{d\text{ prime}\\\text{power}}} \biggl\lceil \frac{r_d(\theta)}{2(d - 1)}\biggr\rceil,
\end{equation}
for the Seifert form $\theta$ of a knot,
it is not necessary to calculate $r_d(\theta)$ for all prime powers $d$; indeed, Livingston showed that
\cref{maxrd} is $0$ if and only if $\Alex{\theta}$ is the product of $n$--th cyclotomic polynomials with $n$ divisible by three distinct primes \cite{livingstonseifert}.
On the other hand, a short calculation yields that \cref{maxrd} can only be greater than $1$ if
$r_d(K) > 2(d-1)$ for some $d$, and because of \cref{eq:rd} this can only happen for $d \leq \deg\Alex{\theta}/2$.
These are the cases for which \cref{thm:cg} goes beyond Kim's \cref{thm:kim}.

One can also ask if \cref{thm:cg} can show that a Seifert form $\theta$ is realized by a knot $K$ without topological genus defect, i.e.~a
knot with $\gtop(K) = g(K)$. One checks that if $\dim \theta > 2$, this can only be accomplished by $d = 2$, namely if $r_2(\theta) \in \{\dim \theta, \dim\theta - 1\}$.
\end{rmk}
\begin{example}
Let us give a concrete example of a Seifert form for which \cref{q:galg=bestclassicalupperbound} is open.
Namely, take $\theta$ to be a Seifert form of the knot $K$, which is $10_{103}$ in Rolfsen's table, given by the following matrix:
\[
\begin{pmatrix}
0 & 0 & 1 & 0 & 0 & 0\\
0 & 0 & 0 & 0 & 0 & -1\\
0 & -2 & -1 & -1 & -1 & 0\\
0 & 0 & 1 & 1 & 0 & 0\\
0 & 0 & 1 & 1 & 1 & -1\\
-1 & -2 & -1 & -2 & -2 & 1
\end{pmatrix}.
\]
The first and second standard basis vector generate an isotropic subgroup of rank $2$, so $t(\theta) \leq 1$.
One computes the signature to be $2$, and so $t(\theta) \geq 1 \Rightarrow t(\theta) = 1$.
On the other hand, the first and third standard basis vector generate an Alexander-trivial subgroup of rank $2$,
so $\ga(\theta) \leq 2$, and moreover $\ua(\theta) = 3$ \cite{knotorious}, which implies $\ga(\theta) \geq 2 \Rightarrow \ga(\theta) = 2$.
Now, the smooth slice genus of $K$ happens to be $1$, and so $\gtop(K) = 1$.
But can $\theta$ be realized by another knot $K'$ with $\gtop(K') = 2$? If $r_2(\theta)$ were at least $3$,
this would follow from \cref{thm:cg}, but $M_d(\theta) = \mathbb{Z}/15 \oplus \mathbb{Z}/5$, whence $r_2(\theta) = 2$.
So we do not know whether $\max_{K\in\mathcal{S}_{\theta}} \gtop(K)$ is 1 or 2.
\end{example}

\begin{rmk}\label{rmk:opt}
We end this section by providing answers to some related questions on optimality of classical slice genus bounds. All of these answers are rather immediate from standard results, but we provide them for completeness.

The algebraic concordance class of a Seifert form yields no upper bound for the topological slice genus:
\begin{equation}\label{i}
\max_{L\in\mathcal{C}_{\theta}} \gtop(L)  = \infty.
\end{equation}
The Seifert form gives only trivial upper bounds for the smooth slice genus:
\begin{align}
\max_{L\in\mathcal{E}_{\theta}} \gs(L) & = \rk(\theta - \theta^{\top}) / 2\qquad\text{and}
\label{ii} \\
\max_{L\in\mathcal{S}_{\theta}} \gs(L) & = \infty.  \label{iii}
\end{align}
\end{rmk}
\begin{proof}[Proof of \cref{i},\cref{ii},\cref{iii}]
Let $\zeta$ be the (algebraically slice) Seifert form given by the matrix
\[
\begin{pmatrix} 0 & 2 \\ 1 & 0 \end{pmatrix}.
\]
Then for any $k \geq 0$, the form $\theta_k = \theta \oplus \zeta^{\oplus k}$ is algebraically concordant to $\theta$.
Moreover, the first homology of the double branched covering of a knot with Seifert form $\theta_k$
has minimum number of generators at least $2k$. Applying \cref{thm:cg} settles \cref{i}.

Statement \cref{ii} may be proven as in \cite{realiseTau}: it is Rudolph's result \cite{Rudolph_83_ConstructionsOfQP1} that every Seifert form $\theta$ may be realized as the Seifert form of a quasipositive Seifert surface $F$. As a consequence of the slice-Bennequin inequality, which Rudolph proved \cite{rudolph_QPasObstruction} building
on Kronheimer and Mrowka's resolution of the Thom Conjecture \cite{KronheimerMrowka_Gaugetheoryforemb},
the smooth slice genus of $\partial F$ equals its three-genus, which is $g(F) = \rk(\theta - \theta^{\top}) / 2$.
This shows \cref{iii} as well, since $\mathcal{S}_{\theta}$ contains Seifert forms of arbitrarily high dimension.
\end{proof}

\section{Reformulation of previously known results in terms of $\ga$}\label{sec:reformulate}
By finding a separating curve with Alexander polynomial 1 on a minimum genus Seifert surface and using Freedman's Theorem \cref{eq:Alex1<=>boundsstddisk}, one can show that $\gtop$ is smaller than the three-dimensional genus. This was used by Rudolph to show that $\gtop(T)<g(T)$ for most torus knots; in fact even $\gtop(T)\leq\frac{9}{10}g(T)$~\cite{Rudolph_84_SomeTopLocFlatSurf}.
Baader used this idea to show that if a minimal genus Seifert surface for a knot $K$ contains an embedded annulus with framing $\pm 1$, then $\widehat{\gtop}(K) = g(K)$ if and only if $|\sigma(K)| = 2g(K)$~\cite{baaderIndef}, a result that we generalized in \cref{prop:charstabgalg<g}.

It turns out that the existence of separating Alexander polynomial 1 knots on Seifert surfaces is completely determined by the Seifert form; compare~\cref{prop:sepcurve}. The following results by Baader, Liechti, McCoy and the authors were proven using some version of this. We present them rewritten in the language of $\ga$, while suppressing the inequalities $\gtop\leq\gst\leq\ga$:\\

\begin{itemize}[leftmargin=1.6em,label={--}]

\item
For all knots $K$ \cite{Feller_15_DegAlexUpperBoundTopSliceGenus}:\qquad
$\displaystyle
\ga(K) \leq \frac{\deg\Alex{K}}{2}.
$\\
\item
For prime homogeneous knots $K$ that are not positive or negative \cite{BaaderLewark_15_Stab4GenusOFAltKnots}:
\[
\widehat{\ga}(K) \leq g(K) - \frac{1}{3}.
\]

\item There is an infinite family of 2--bridge knots $K$ satisfying \cite{FellerMcCoy_15}
\[
\ga(K) < \gs(K) = g(K).
\]

\item The algebraic genus of torus links satisfies \cite{BaaderFellerLewarkLiechti_15}
\[
\frac{1}{2} \leq \lim_{p,q\to\infty} \frac{\ga(T_{p,q})}{g(T_{p,q})} \leq \frac{3}{4}
\]
and for $p \geq q \geq 3$, and $(p,q)\not\in\{(3,3),(4,3),(5,3),(6,3),(4,4)\}$:
\[
\frac{1}{2} \leq \frac{\ga(T_{p,q})}{g(T_{p,q})} \leq \frac{6}{7}=\frac{\ga(T_{8,3})}{g(T_{8,3})}.
\]

\item For all prime knots with up to 11 crossings, one has \cite{mccoylewark}
\[
\gtop(K) = \min\{\ga(K), \gs(K)\}.
\]

\item For positive braid knots $K$ with $\sigma(K) < 2g(K)$ one has \cite{LiechtiBraids}
\[
\ga(K) < g(K).
\]
\end{itemize}

\bibliographystyle{myamsalpha}
\bibliography{peterbib}
\end{document}